\documentclass[12pt]{article}
\pdfoutput=1
\usepackage[utf8]{inputenc}
\usepackage[T1]{fontenc}
\usepackage{lmodern}
\usepackage{microtype}
\usepackage{fullpage} 
\usepackage{cite}

\usepackage{eurosym}
\usepackage{epic,latexsym,amssymb,amsmath}
\usepackage{tikz}
\usepackage{hyperref}
\usepackage{pstricks}
\usepackage{pst-plot}
\usepackage{pst-node}
\usepackage{amsmath,amsfonts,latexsym, amssymb}
\usepackage[mathscr]{euscript}
\usepackage[absolute]{textpos}
\usepackage{mathrsfs}
\usepackage{graphics}
\usepackage{float}
\usepackage{enumerate}
\usepackage{bm}
\usepackage[normalem]{ulem}
\usepackage{pstricks,pst-node,pst-plot} %
\usepackage{color} %
\usepackage{geometry}
\newtheorem{thm}{Theorem}[section]

\newtheorem{lemma}[thm]{Lemma}

\newtheorem{obs}[thm]{Observation}
\newtheorem{claim}[thm]{Claim}

\newtheorem{definition}[thm]{Definition}

\newcommand{\ch}{\text{ch}}

\newcommand{\oururl}{\url{http://lidicky.name/pub/flexibility}}

\usetikzlibrary{positioning, fit, calc}
\usetikzlibrary {arrows.meta}
\tikzset{vtx/.style={inner sep=1.7pt, outer sep=0pt, circle, fill,draw}}
\tikzset{bndry/.style={inner sep=2pt, outer sep=0pt,fill=white,draw,shape=circle}}
\tikzset{gedge/.style={solid,color=black,line width=1.2pt,opacity=0.75}} 
\tikzset{dashedge/.style={dashed,color=black,line width=1pt,opacity=1}} 
\tikzset{vtxNoFIX/.style={inner sep=3pt, outer sep=0pt,fill=black,draw}}

\def\baseConfigurationdiamond#1#2#3#4{
\node[vtx,label=above:#1] (u) at (-1.5,0) {};
\node[vtx,label=above:#2] (v) at (1.5,0) {};
\node[vtx,label=left:#3] (x) at (0,1.5) {};
\node[vtx,label=left:#4] (y) at (0,-1.5) {};
\draw [gedge] (u) -- (v)  (u) -- (x)  (u) -- (y)  (v) -- (x) (v) -- (y);
}

\newcommand{\confgurationC}{
    \baseConfigurationdiamond{}{}{}{}
    
    \draw[gedge] (x) -- ++(120:0.5);
    \draw[gedge] (x) -- ++(60:0.5);
    
    \draw[gedge] (y) -- ++(240:0.5);
    \draw[gedge] (y) -- ++(300:0.5); 
    \draw[gedge] (y) -- ++(270:0.5);
}

\newcommand{\confgurationD}{
    \baseConfigurationdiamond{}{}{}{}
    
    \draw[gedge] (x) -- ++(90:0.5);
    
    \draw[gedge] (v) -- ++(0:0.5);
    
    \draw[gedge] (y) -- ++(240:0.5);
    \draw[gedge] (y) -- ++(300:0.5); 
}

\newcommand{\confgurationE}{
    \baseConfigurationdiamond{}{}{}{}
    
    \draw[gedge] (x) -- ++(90:0.5);
    
    \draw[gedge] (u) -- ++(150:0.5);
    \draw[gedge] (u) -- ++(210:0.5);    

    \draw[gedge] (v) -- ++(30:0.5);
    \draw[gedge] (v) -- ++(-30:0.5);
   
    \draw[gedge] (y) -- ++(270:0.5);
}

\newcommand{\confgurationF}{
    \baseConfigurationdiamond{}{}{}{}
    
    \draw[gedge] (x) -- ++(90:0.5);
    
    \draw[gedge] (u) -- ++(150:0.5);
    \draw[gedge] (u) -- ++(210:0.5);    

    \draw[gedge] (v) -- ++(0:0.5);
   
    \draw[gedge] (y) -- ++(240:0.5);
    \draw[gedge] (y) -- ++(300:0.5);
}

\newcommand{\confgurationG}{
    \baseConfigurationdiamond{}{}{}{}
    
    \draw[gedge] (y) -- ++(240:0.5);
    \draw[gedge] (y) -- ++(300:0.5);
    \draw[gedge] (y) -- ++(270:0.5);
}
\def\baseConfigurationtriangle#1#2#3{
\node[vtx,label=left:#1] (w) at (1,0) {};
\path (w) ++(60:2) node (u)[vtx,label=right:#2] {};
\path (w) ++(0:2) node (v)[vtx,label=right:#3] {};

\draw[gedge] (u) -- (v);
\draw[gedge] (w) -- (u);
\draw[gedge] (w) -- (v);
}

\newcommand{\baseConfigurationtriangleUnlabeled}{
\node[vtx, label = left:$ $] (w) at (1,0) {};
\path (w) ++(60:2) node (u)[vtx, label = left:$ $] {};
\path (w) ++(0:2) node (v)[vtx, label = right:$ $] {};

\draw[gedge] (u) -- (v);
\draw[gedge] (w) -- (u);
\draw[gedge] (w) -- (v);
}
\newcommand{\baseConfigurationsquare}{
\node[vtx] (u) at (-1,1) {};
\node[vtx] (v) at (-1,-1) {};
\node[vtx] (x) at (1,1) {};
\node[vtx] (y) at (1,-1) {};

\draw[gedge] (u) -- (v);
\draw[gedge] (x) -- (u);
\draw[gedge] (x) -- (y);
\draw[gedge] (y) -- (v);
}

\newcommand{\baseConfigurationsquareUnlabeled}{
\node[vtx, label = left:$ $] (u) at (-1,1) {};
\node[vtx, label = left:$ $] (v) at (-1,-1) {};
\node[vtx, label = right:$ $] (x) at (1,1) {};
\node[vtx, label = right:$ $] (y) at (1,-1) {};

\draw[gedge] (u) -- (v);
\draw[gedge] (x) -- (u);
\draw[gedge] (x) -- (y);
\draw[gedge] (y) -- (v);
}

\newcommand{\chargepath}[5]{
#1 edge[#2] node[#3] {\footnotesize{#4}} #5
}

\renewcommand{\epsilon}{\varepsilon}
\newcommand{\brm}[1]{\operatorname{#1}}
\DeclareMathOperator\df{:=}

\newenvironment{proof}[1][Proof]{\textbf{#1.} }{\ \rule{0.5em}{0.5em}}

\begin{document}
\begin{textblock}{20}(0, 12.5)
\includegraphics[width=40px]{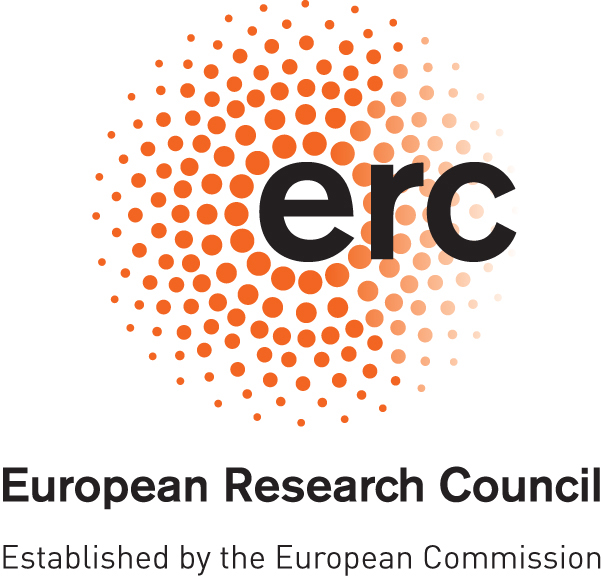}%
\end{textblock}
\begin{textblock}{20}(-0.25, 12.9)
\includegraphics[width=60px]{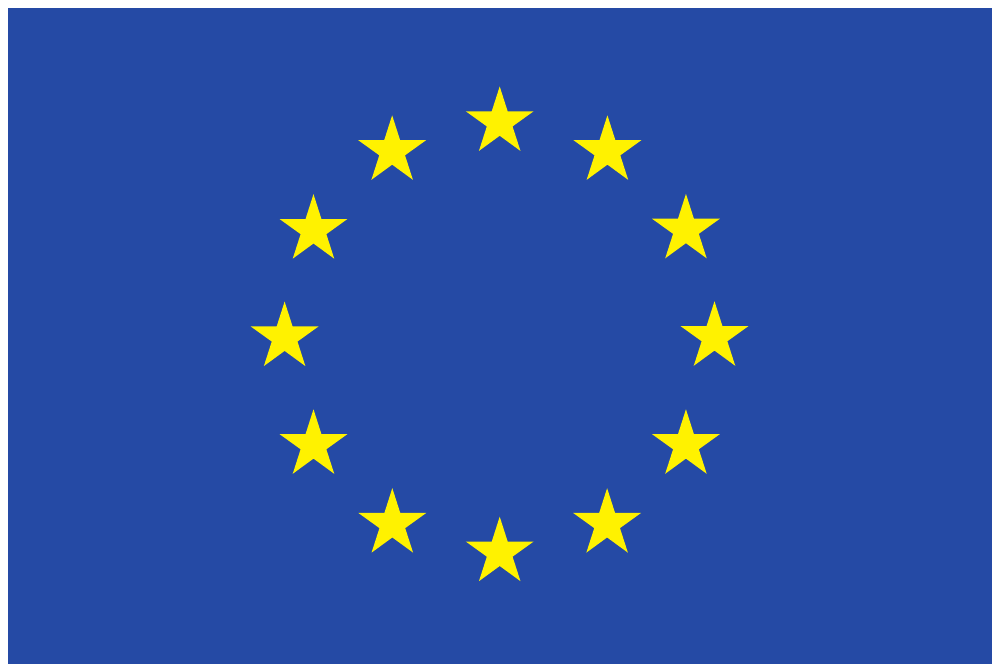}%
\end{textblock}

\tikzset{edge_color0/.style={color=black,line width=1.2pt,opacity=0.5}} 
\tikzset{edge_color1/.style={color=black,dashed,line width=1.2pt,opacity=0.5}} 
\tikzset{unlabeled_vertex/.style={inner sep=1.7pt, outer sep=0pt, circle, fill}}

\title{On Weak Flexibility in Planar Graphs}
\author{Bernard Lidick\'y\thanks{Department of Mathematics, Iowa State University. Ames, IA, USA. E-mail: \texttt{lidicky@iastate.edu} Supported in part by NSF grant DMS-1855653.}
\and
Tom\'a\v{s} Masa\v{r}\'{\i}k\thanks{Faculty of Mathematics, Informatics and Mechanics, University of Warsaw, Poland \& Department of Mathematics, Simon Fraser University, Burnaby, BC, Canada. E-mail: \texttt{masarik@kam.mff.cuni.cz}
T.~Masařík received funding from the European Research Council (ERC) under the European Union’s Horizon 2020 researchand innovation programme Grant Agreement 714704. He completed a part of this work while he was a postdoc at Simon Fraser University in Canada.
}
\and 
Kyle Murphy\thanks{Department of Mathematics, Iowa State University. Ames, IA, USA. E-mail: \texttt{kylem2@iastate.edu}}
\and 
Shira Zerbib\thanks{Department of Mathematics, Iowa State University. Ames, IA, USA. E-mail: \texttt{zerbib@iastate.edu} Supported by NSF grant DMS-1953929.}
}\date{}

\maketitle

\begin{abstract}
Recently, Dvořák, Norin, and Postle introduced flexibility as an extension of list coloring on graphs [JGT 19']. In this new setting, each vertex $v$ in some subset of $V(G)$ has a request for a certain color $r(v)$ in its list of colors $L(v)$. The goal is to find an $L$ coloring satisfying many, but not necessarily all, of the requests.

The main studied question is whether there exists a universal constant $\epsilon >0$ such that any graph $G$ in some graph class $\mathcal{C}$ satisfies at least $\epsilon$ proportion of the requests. 
More formally, for  $k  > 0$ the goal is to prove that for any graph $G \in \mathcal{C}$ on vertex set $V$, with any list assignment $L$ of size $k$ for each vertex, and for every $R \subseteq V$ and a request vector $(r(v): v\in R, ~r(v) \in L(v))$, there exists an $L$-coloring of $G$ satisfying at least $\epsilon|R|$ requests. 
If this is true, then  $\mathcal{C}$ is called {\em $\epsilon$-flexible for lists of size $k$}.

Choi et al.\ [arXiv 20'] introduced the notion of {\em weak flexibility}, where $R = V$.
We further develop this direction by introducing a %
tool to handle weak flexibility.
We demonstrate this new tool by showing that for every positive integer $b$ there exists $\epsilon(b)>0$ so that the class of planar graphs without $K_4, C_5 , C_6 , C_7, B_b$ is  weakly $\epsilon(b)$-flexible for lists of size $4$ (here  $K_n$, $C_n$ and $B_n$ are the complete graph, a cycle, and a book on $n$ vertices, respectively).
We also show that the class of planar graphs without $K_4, C_5 , C_6 , C_7, B_5$ is $\epsilon$-flexible  for lists of size $4$.
The results are tight as these graph classes are not even 3-colorable.

\end{abstract}

\section{Introduction}

A \emph{$k$-coloring} of a graph $G$ is a function $f: V(G) \to S$, where $|S| = k$. The elements of $S$ are often called \emph{colors}. A $k$-coloring of $G$ is called \emph{proper} if adjacent vertices are assigned different colors. 
Suppose that for  each vertex $v$ in $G$, we gave $v$ a list $L(v)$ of available colors. 
A \emph{list coloring} of a graph $G$ is a proper coloring of $G$ where each vertex $v$ is assigned a color from $L(v)$. In particular, for two distinct vertices $u$ and $v$, $L(u)$ and $L(v)$ might be different. 
A graph is {\it $k$-choosable} if every assignment $L$ of at least $k$ colors to each vertex guarantees an $L$-coloring.
The \emph{choosability} of a graph $G$ is the minimum $k$ such that $G$ is $k$-choosable.

In many applications of list coloring, such as scheduling, some vertices may have preferences which are not directly captured by the lists themselves. 
For example, a professor may be willing to teach classes X,Y, or Z but prefers to teach X. Ideally, the scheduler can satisfy the specific requests of each professor, but it is often the case that they cannot. The goal is then to satisfy as many requests as possible. This idea motivates the following definitions.  

A \emph{weighted request} is a function $w$ that assigns a nonnegative real number to each pair $(v,c)$ where $v\in V(G)$ and $c\in L(v)$.  
For $\varepsilon>0$, we say that $w$ is \emph{$\varepsilon$-satisfiable} if there exists an $L$-coloring $\varphi$ of $G$ such that
\[
 \sum_{v\in V(G)} w(v,\varphi(v))\ge\varepsilon\cdot \sum_{v\in V(G),c\in L(v)} w(v,c).
\]

The unweighted variant is defined as follows.
A \emph{request} for a graph $G$ with a list assignment $L$ is a function $r$ with domain $\brm{dom}(r)\subseteq V(G)$ such that $r(v)\in L(v)$ for all $v\in\brm{dom}(r)$.
In the special case that each vertex requests a color, i.e., $\brm{dom}(r)=V(G)$, we call such a request \emph{widespread}.
Analogously, for $\varepsilon>0$, a request $r$ is \emph{$\varepsilon$-satisfiable} if there exists an $L$-coloring $\varphi$ of $G$ such that at least $\epsilon|\brm{dom}(r)|$ vertices $v$ in $\brm{dom}(r)$ receive color $r(v)$. 
We say that a graph $G$ with list assignment $L$ is \emph{$\varepsilon$-flexible}, \emph{weakly  $\varepsilon$-flexible}, or \emph{weighted $\varepsilon$-flexible} if every request, widespread request, or weighted request, respectively, is $\varepsilon$-satisfiable.
Note that weak flexibility does not make sense in the weighted setting since one can set some weights to $0$ to turn off the requests for these vertices. 
If $G$ is (weighted/weakly) $\varepsilon$-flexible for every list assignment with lists of length $k$, we say that $G$ is {\it (weighted/weakly) $\varepsilon$-flexible for lists of size $k$.} 
Note that for $k$-colorable graphs, if the lists are exactly the same the problem becomes trivial as by permuting the colors we can achieve $1\over k$-flexibility~\cite{dvoraknorin}.

The concept of $\varepsilon$-flexibility was introduced by Dvořák, Norin, and Postle~\cite{dvoraknorin}.
Subsequently, it was studied for various sub-classes of planar graphs, e.g., triangle-free~\cite{dvorak}, girth six~\cite{dvorak2}, or $C_4$-free~\cite{masarik}.
Graphs of bounded maximum degree were subsequently characterized in terms of flexibility~\cite{Bradshaw}.

A central notion in graph coloring is that of {\em reducible configurations}, which are local subgraphs that cannot appear in a smallest counterexample because their presence  implies that the graph can be colored from a smaller subgraph by induction. 
Reducible configurations for flexibility are slightly more delicate as we explain in Section~\ref{sec:flexissues}.
Recently, Choi et al.~\cite{choiclemen} proposed a strengthened tool (see Lemma~\ref{lem:weighted} below) for designing reducible configurations for flexibility. 
The authors of~\cite{choiclemen} also introduced the notion of weak flexibility defined above.
They demonstrated that the weak setting allows one to create stronger reducible configurations.

We further develop this direction by strengthening the tools for handling weak flexibility; see Lemma~\ref{lem:weak} in Section~\ref{sec:method}.
We exhibit our new tool by showing the following results for subclasses of planar graphs.

For an integer $n\ge 3$ let 
 $B_n$  denote the book on $n$ vertices, i.e., the graph consisting of $n-2$ triangles sharing an edge.  
Let $C_n$ and $K_n$  denote a cycle and a clique on $n$ vertices, respectively.
Given a set of graphs $\mathcal F$ and a graph $H$, we say that $H$ is \emph{$\mathcal{F}$-free} if there is no subgraph of $H$ isomorphic to any of the graphs in $\mathcal{F}$.

\begin{thm}\label{main567}
There exists $\varepsilon  > 0$ such that every planar $\{K_4,C_5,C_6,C_7,B_5\}$-free graph  is weighted $\varepsilon$-flexible for lists of size $4$.
\end{thm}

\begin{thm}\label{main567weak}
There exists $\varepsilon=\varepsilon(b) > 0$ such that every planar $\{K_4,C_5,C_6,C_7,B_b\}$-free graph is weakly $\varepsilon$-flexible for lists of size $4$.
\end{thm}

The results in Theorems~\ref{main567} and~\ref{main567weak} are tight as in general such graphs are not even 3-colorable. This is exemplified by the construction in Figure~\ref{f:construction}. This construction implies:

\begin{obs}\label{o:construction}
For every $\ell,b \geq 5$ exists a $\{K_4,B_b\}$-free planar graph $G$ that does not contain any cycle $C_k$ of length $5\le k\le \ell$, such that $G$ is a not $3$-colorable.
\end{obs}

\begin{figure}[bt]\label{f:construction}
\begin{center}
\begin{tikzpicture}[scale = 1.5]

\node[vtx,label = below:\footnotesize{$1$}] (a) at (0,0) {};
\node[vtx,label = above:\footnotesize{$2$}] (b) at (0.5,0.5) {};
\node[vtx,label = below:\footnotesize{$3$}] (c) at (0.5,-0.5) {};
\node[vtx,label = below:\footnotesize{$1$}] (d) at (1,0) {};
\draw[gedge] (a)--(b) (a)--(c) (d)--(b) (d)--(c) (b)--(c);
\node[vtx,label = above:\footnotesize{$2$}] (f) at (1.5,0.5) {};
\node[vtx,label = below:\footnotesize{$3$}] (g) at (1.5,-0.5) {};
\node[vtx,label = below:\footnotesize{$1$}] (h) at (2,0) {};
\draw[gedge] (d)--(f) (d)--(g) (h)--(f) (h)--(g) (f)--(g);
\node[vtx,label = above:\footnotesize{$2$}] (i) at (2.5,0.5) {};
\node[vtx,label = below:\footnotesize{$3$}] (j) at (2.5,-0.5) {};
\node[vtx,label = below:\footnotesize{$1$}] (k) at (3,0) {};
\draw[gedge] (h)--(i) (h)--(j) (k)--(i) (k)--(j) (i)--(j);
\node[vtx,label = above:\footnotesize{$2$}] (l) at (3.5,0.5) {};
\node[vtx,label = below:\footnotesize{$3$}] (m) at (3.5,-0.5) {};
\node[vtx,label = below:\footnotesize{$1$}] (n) at (4,0) {};
\draw[gedge] (k)--(l) (k)--(m) (n)--(l) (n)--(m) (l)--(m);
\node[vtx,label = above:\footnotesize{$2$}] (o) at (4.5,0.5) {};
\node[vtx,label = below:\footnotesize{$3$}] (p) at (4.5,-0.5) {};
\node[vtx,label = below:\footnotesize{$?$}] (q) at (5,0) {};
\draw[gedge] (n)--(o) (n)--(p) (q)--(o) (q)--(p) (o)--(p);
\draw[gedge] (0,0) to [out = 90, in = 180] (2.5,2) to [out = 0, in = 90] (5,0);
\end{tikzpicture}
\end{center}
  \caption{A construction proving Observation~\ref{o:construction} with an attempt for a $3$-coloring that fails. }
\end{figure}
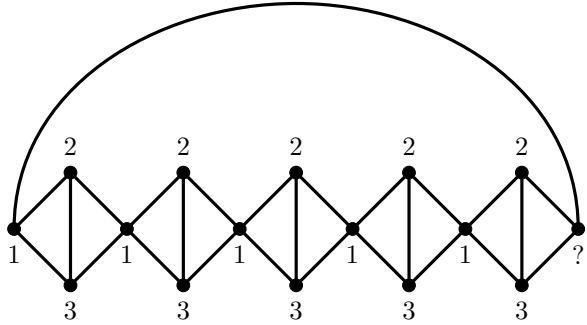

Furthermore, our results follow a recent line of research trying to narrow the gap between known degeneracy upper-bounds and choosability lower-bounds, in particular on subclasses of planar graphs, as is described below. 
We say that a graph $G$ is \emph{$d$-degenerate} if each induced subgraph of $G$ contains a vertex of degree at most $d$.
It is  easy to observe that $d$-degenerate graphs are $(d+1)$-choosable.
A similar statement holds for flexibility as well:
in~\cite{dvoraknorin} it was proved that $d$-degenerate graphs with lists of size $d+2$ are weighted $\epsilon$-flexible.
Therefore, as $C_5$-free planar graphs are 3-degenerate~\cite{wanglih}, they are $\epsilon$-flexible for lists of size 5.
The same is true for $C_6$-free planar graphs~\cite{fijavzjuvan}.
For $C_3$-free graphs, Dvo\v{r}\'{a}k, Masa\v{r}\'{\i}k, Mus\'{\i}lek, and Pangr\'{a}c~\cite{dvorak} showed that they are weighted $\epsilon$-flexible for lists of size 4 and that this the result is tight.
Surprisingly, the discharging proof in \cite{dvorak} is quite involved compared to the easy observation that $C_3$-free planar graphs are $3$-degenerate, which implies $4$-choosability.
An analogous result holds for  $\{C_3,C_4,C_5\}$-free graphs, where list of size $3$ are sufficient for weighted  $\epsilon$-flexibility and the result is tight~\cite{dvorak2}.

When only $C_4$ is forbidden, Masa\v{r}\'{\i}k~\cite{masarik} proved that lists of size 5 are sufficient for weighted $\epsilon$-flexibility.
However, it is unknown whether the result is tight as those graphs are 4-choosable~\cite{lamxu} (but not necessarily 3-degenerate).
There were attempts to bring down the list size to 4 but so far only partial results are known in this direction: planar graphs that do not contain $C_4$ and $C_3$ at distance at most 1~\cite{choiclemen} or $\{C_4,C_5\}$-free planar graphs \cite{yangyang}.
See \cite[Table~1]{choiclemen} for a comprehensive overview of known results for various subclasses of planar graphs.
Our results aim to improve this narrow gap as they show that lists of size 4 are sufficient even for planar graphs in which some copies of $C_3$ and $C_4$ are allowed.

\section{Methods - informal discussion}~\label{sec:flexissues}

The purpose of this section is to informally describe some of the difficulties one faces when trying to  extend a list-coloring proof to a flexibility proof. This discussion serves as the intuition behind the formal definitions in the next section.

As in previous related papers mentioned above,  we use the discharging method to obtain our results. For an introduction to the discharging method see~\cite{discharging}.
A typical discharging proof that a graph $G$ is $L$-list-colorable gives a list of unavoidable \emph{reducible configurations}, which are subgraphs of $G$ that cannot appear in a minimal counterxample.
The goal is to decompose $G$ into subgraphs $R_1,\ldots,R_N$ such that $R_i$ is a reducible configuration in $G[R_i \cup \dots \cup R_N]$ (this will be defined as a \emph{resolution} later), so that
  any $L$-coloring of $G[R_{i+1}\cup \dots \cup R_N]$ can be extended to an $L$-coloring of $G[R_{i}\cup \dots \cup R_N]$. 
Extending the coloring in a descending order from $R_N$ to $R_1$ gives an $L$-coloring of $G$.

When requests are introduced, this method becomes more difficult.
To explain this, assume 
for simplicity that every vertex has a request.
If we manage to accommodate one request from each $R_i$ and each $R_i$ has at most $b$ vertices, then we would satisfy $n/b$ requests, showing that $G$ is $\varepsilon$-flexible for $\varepsilon=1/b$, and our job would be done. However, this is not necessarily possible. Indeed, 
suppose $v \in R_{i}$ has some request $r(v)$ and let $\varphi$ be an $L$-coloring of $G[R_{i+1} \cup \dots \cup R_{N}]$.
Suppose further that $v$ has one neighbor $u$ in $R_{i+1} \cup \dots \cup R_{N}$. If $\varphi(u)=r(v)$, there is no way to simply extend $\varphi$ and accommodate the request of $v$. 
Thus more changes to $\varphi$, such as recoloring $u$, would have to occur to accommodate $r(v)$.
In addition to this issue, it may also be the case that $r(v)$ cannot be satisfied because $R_i$ itself prevents it.

This means that reducible configurations for flexibility need to provide slightly more freedom in the colorings they allow.
The easier problem to deal with is that $r(v)$ cannot be satisfied because $R_i$ itself prevents it.
This can be patched by adding a requirement that for any one vertex $x$ in $R_i$, the coloring $\varphi$ extends to $R_i$ even if $x$ has a list of size $1$ after removing the colors of already colored neighbors of $x$ in  $R_{i+1}\cup \cdots \cup R_N$.
For $v$, this would be used in case $L(v)=\{r(v)\}$ and $\varphi(u) \neq r(v)$.
This requirement will be called (FIX) in the formal definitions.

The problem occurring when $r(v)$ cannot be satisfied because its neighbor $u$ in $R_{i+1} \cup \dots \cup R_{N}$  is already colored $r(v)$ is more complicated to solve. The idea is the following.
Instead of constructing just one $L$-coloring $\varphi$, one needs to construct $L$-colorings $\varphi_1,\ldots,\varphi_\ell$ and in some of them, $u$  gets colored by a color different  than $r(v)$.
Then  $\varphi_1,\ldots,\varphi_\ell$ can be extended to  $\varphi'_1,\ldots,\varphi'_{\ell'}$, where $r(v)$ is satisfied in some of them. 
At the end, this process gives a set of $L$-colorings of $G$ and at least one of them  satisfies a positive fraction of the requests.
Formally, this is done by creating a probability distribution on $L$-colorings of $G$.

In order to make this idea work, there must be a sufficient variety of proper colorings for each reducible configuration.
In our example, if we want to color $v$ by $r(v)$, we cannot use $r(v)$ on $u$. We need to address this when we are coloring $u$ and remove $r(v)$ from its list. 
Further, we would need to do this for each neighbor of $v$ in $R_{i+1}\cup \dots \cup R_{N}$. 
This is achieved in the following way.
When we are $L$-coloring $R_i$, we look at all subsets $I \subseteq V(R_i)$ of vertices that could form a neighborhood of a vertex in $R_1 \cup \ldots \cup R_{i-1}$, i.e. in the set of not yet colored vertices. Individually for each $I$, we show that any proper $L$-coloring $\varphi$ of $R_{i+1}\cup \dots \cup R_{N}$ can still be extended to $R_i$ even if we decrease the sizes of the lists of vertices in $I$ by $1$.
This will be called (FORB) in the formal definitions.

To summarize, the reducible configurations for flexibility must have size bounded by a constant, any one vertex can be precolored (FIX), and for different subsets of vertices, reducing their lists sizes by $1$ does not break the extendability of the coloring (FORB).

The main feature of weak flexibility is that instead of demanding in (FIX) that ``any one vertex in $R_i$ can be precolored'', it is enough to ask for
``at least one vertex in $R_i$ can be precolored''. 
It is then easier to satisfy this version of (FIX). 

We introduce an additional trick, where we ask ``at least one vertex in $R_i$ with  few outside neighbors can be precolored''. This makes satisfying (FIX) more difficult since the the reducible configurations needs a vertex of small degree but it helps a lot with checking (FORB).

\section{Methods - definitions and lemmas}~\label{sec:method}

We  use  some of the notation and  tools introduced  in \cite{choiclemen, dvoraknorin, dvorak, dvorak2}. In particular, 
our definitions are quite similar to those used in \cite{choiclemen}.

Let $1_I$ denote the characteristic function of $I$, i.e., $1_I(v)=1$ if $v\in I$ and $1_I(v)=0$ otherwise.
Let $G$ be a graph.
Given a function $f:V(G)\to\mathbb{Z}$ and a vertex $v\in V(G)$, let $f\downarrow v$ denote the function satisfying $(f\downarrow v)(w)=f(w)$ for $w\neq v$ and $(f\downarrow v)(v)=1$.
We will use $f\downarrow v$ to indicate that the list size at vertex $v$ has been reduced to $1$. In other words, $f\downarrow v$ means that $v$ has been ``precolored''.
A list assignment $L$ is an \emph{$f$-assignment} if $|L(v)|\ge f(v)$ for all $v\in V(G)$.
We will let $\deg_G$ be the function from $V(G)$ to $\mathbb{Z}$ which maps each vertex to its degree.
If $X \subset V(G)$, then we let $\deg_X$ equal $\deg_{G[X]}$, where $G[X]$ is the induced subgraph of $G$ consisting of the vertices in $X$. 
 
 Given a set of graphs $\mathcal F$ and a graph $H$, a set $I\subseteq V(H)$ is \emph{$\mathcal{F}$-free} if the graph $H$ together with one additional vertex $u$ adjacent to all of the vertices in $I$ does not contain any subgraph isomorphic to a graph in $\mathcal{F}$.
Throughout the following definitions, let  $H$ be an induced subgraph of a graph $G$, let $\mathcal{F}$ be a set of graphs, and let $k$ be a positive integer.

\begin{definition}[$(\mathcal{F},k)$-boundary-reducibility]
\label{reduce} 
 We say that 
$H$ is an \emph{$(\mathcal{F},k)$-boundary-reducible} subgraph if there exists a set $R \subseteq V(H)$ such that $R \neq \emptyset$ and
\begin{itemize}
  \item[\emph{(FIX)}] for every $v\in R$, $H[R]$ is $L$-colorable for every (($k-\deg_G+\deg_R)\downarrow$ $v$)-assignment $L$, and 
  \item[\emph{(FORB)}] for every $\mathcal{F}$-free set $I \subseteq R$ of size at most $k-2$, $H[R]$ is $L$-colorable for every ($k-\deg_G+\deg_R-1_I$)-assignment $L$. 
\end{itemize}

\end{definition}

\begin{definition}[weak $(\mathcal{F},k)$-boundary-reducibility]\label{weak_reduce}
  We say that $H$ is \emph{weakly $(\mathcal{F},k)$-boundary-reducible} if it satisfies {\em (FORB)} and there exists at least one vertex $v$ satisfying {\em (FIX)} from Definition~\ref{reduce}. In this case, we denote $v$ by $\mathrm{Fix}(H)$.
\end{definition}

In both of the preceding definitions, we will occasionally refer to the set $V(H)\setminus R$ as the \textit{boundary} of the configuration and the set $R$ as the \emph{reduced part} of the configuration.  
Note that (FORB) in particular implies that $\deg_G-\deg_R\le k-2$ for all $v\in R$.

\begin{definition}[$(\mathcal{F},k,b)$-resolution]\label{def:resolution}
Let $G$ be an $\mathcal{F}$-free graph with lists of size $k$.
An \emph{$(\mathcal{F},k,b)$-resolution} of $G$ is a set $\{G_0, G_1, \dots, G_M\}$ of 
 subgraphs of $G$ such that for $i \geq 1$, $H_i$ is an induced $(\mathcal{F},k)$-boundary-reducible subgraph of $G_{i-1}$ with reduced part $R_i$ and 
 \[
  G_i\df G- \bigcup_{j=1}^{i} {R_j}.
\]
Additionally, for each $i \geq 1$ $|R_i|\le b$ and $G_{M}$ is itself a $(\mathcal{F},k)$-boundary-reducible graph with empty boundary and order at most $b$. 
For technical reasons, let $G_{M +1}\df\emptyset$.

\end{definition}

A \emph{weak $(\mathcal{F},k,b)$-resolution} is defined analogously, save that it uses weak $(\mathcal{F},k)$-boundary-reducibility in the place of $(\mathcal{F},k)$-boundary-reducibility.

The following lemma is the main tool  we use for proving weighted $\epsilon$-flexibility.

\begin{lemma}[Lemma~13~in~\cite{choiclemen}]\label{lem:weighted}
  For  integers $k\ge 3$ and $b\ge 1$ and for a set $\mathcal{F}$ of forbidden subgraphs, 
  let $G$ be an $\mathcal{F}$-free graph with an $(\mathcal{F},k,b)$-resolution. 
 Then there exists an $\varepsilon>0$ such that $G$ is weighted $\varepsilon$-flexible for  lists of size $k$.
 Furthermore, if the request is widespread and $G$ has a weak $(\mathcal{F},k,b)$-resolution, 
 then $G$ is weakly $\left(\varepsilon\cdot \frac{1}{b}\right)$-flexible for lists of size $k$.
\end{lemma}

For the proof of Theorem \ref{main567weak} we prove a stronger version of Lemma~\ref{lem:weighted} tailored to the setting of weak flexibility.
For this, we define new ``enhanced'' versions of weak $(\mathcal{F},k)$-boundary-reducibility and of a weak $(\mathcal{F},k,b)$-resolution.
We will now require $\text{Fix}(H)$ to contain only vertices $v$ satisfying $\deg_G (v)-\deg_R (v)\le k-3$.
This change will allows us to consider smaller sets for the (FORB) condition.

\begin{definition}[enhanced weak $(\mathcal{F},k)$-boundary-reducibility]
\label{new_weak_reduce}
  A graph $H$ is  \emph{enhanced weakly $(\mathcal{F},k)$-boundary-reducible}  if there exist non-empty sets $\mathrm{Fix}(H)\subseteq R \subseteq V(H)$  such that
  \begin{itemize}
    \item[{\em (FIX)}] for every $v\in \mathrm{Fix}(H)$,  $\deg_G(v)-\deg_R(v)\le k-3$ and  $H[R]$ is $L$-colorable for every (($k-\deg_G+\deg_R)\downarrow$ $v$)-assignment $L$, and 
    \item[{\em (FORB)}] for every  $\mathcal{F}$-free set $I \subseteq R$ of size at most $k-3$, $H[R]$ is $L$-colorable for every ($k-\deg_G+\deg_R-1_I$)-assignment $L$.
  \end{itemize}

\end{definition}

Before proceeding further, observe that (FORB) in the enhanced version is easier to check because  $I$ is of  size at most $k-3$, instead of $k-2$ in the non-enhanced version. However, (FIX) in the enhanced version has an additional restriction on the degree of vertices in $\text{Fix}(H)$, which makes it more difficult to satisfy. 
Note that in general, the (FORB) condition on a single vertex $v$ implies $\deg_G-\deg_R\leq k-2$. 
However for vertices in $\text{Fix}(H)$, the (FIX) condition implies $\deg_G-\deg_R\leq k-3$.
In particular, a vertex of degree $k-2$ is no longer reducible under the enhanced definition.
We overcome this obstacle by allowing $(k-2)$-vertices in a resolution under special conditions.
Forbidding books $B_\ell$ helps with satisfying these special conditions. 
By doing this we can have both: a vertex of degree $k-2$ is reducible in our setting, and in addition we obtain subgraphs $H$ that are reducible under the enhanced weak $(\mathcal{F},k)$-boundary-reducibility definition, given that certain special circumstances occur. 
This rather technical improvement helps substantially in reducing the complexity of the analysis of the discharging process for the graph classes studied in this paper.
Note that further generalization of this idea may be possible, but for lack of use in this paper we will not aim to formulate this in the full generality.
\medskip

For a subgraph $H$ of a graph $G$, let $N_G(H)$ be the induced subgraph of $G$ on all neighbors of the vertices in $H$.

If $G$ is a graph satisfying the conditions of Definition~\ref{new_weak_reduce} and $I \subseteq R$ is an $\mathcal{F}$-free set of size $k-2$ so that $G[R]$ is $L$-colorable for every ($k-\deg_G+\deg_R-1_I$)-assignment $L$, then we call $I$ \emph{loose}.

Let $G$ be a graph, $H$ its subgraph and $v \in V(G-H)$.
We say that $v$ is \emph{$H$-tight} if
$\deg_G(v) = k-2$, $N_G(v) \subseteq V(H)$,
and $N_G(v)$ is not loose in $H$.

\begin{definition}[enhanced weak $(\mathcal{F},k,b,\beta)$-resolution]
  \label{def:new_resolution}
  Let $G$ be an  $\mathcal{F}$-free graph with lists of size $k$.
An \emph{enhanced weak $(\mathcal{F},k,b,\beta)$-resolution} of $G$ is a set $\{G_0, G_1, \dots, G_M\}$ of 
subgraphs of $G$, such that all the following three conditions hold: 
\begin{enumerate}
    \item For $1\le i \le M$, there exists a subgraph $H_i$ of $G_{i-1}$ satisfying that 
\begin{itemize}
  \item  $H_i$ is an induced enhanced weak $(\mathcal{F},k)$-boundary-reducible subgraph of $G_{i-1}$ with %
  reducible part $R_i$ such that $|R_i|\le b$,  or
  \item  $H_i$ is an induced weak $(\mathcal{F},k)$-boundary-reducible subgraph of $G_{i-1}$ with %
    reducible part $R_i$, such that $|R_i|\le b$ and for all $v\in Fix(H_i)$ either  $|N_{G_{i-1}}(v)\cap H_j|\le k-3$ or $N_{G_{i-1}}(v)\cap H_j$ is a loose set in $H_j$ for all $j>i$,  or
  \item $H_i$ is a single vertex with  $\deg_{G_{i-1}}(v)=k-2$.
\end{itemize} 
\item For every $1\le i \le M-1$,
  \[
  G_i\df G- \bigcup_{j=1}^{i} {R_j},
\]
$G_{M}$ is a weak $(\mathcal{F},k)$-boundary-reducible graph with empty boundary and order at most $b$, $G_{M +1}\df\emptyset$, and $H_{M+1}\df G_M$. 

\item The following is satisfied:
\begin{itemize}
  \item [{\em (TIGHT)}] For every  $1\le j\le M$, there are at most $\beta$ different $H_j$-tight vertices $v_i$, where $V(H_i)=\{v_i\}$ with $i < j$. \end{itemize}
\end{enumerate}

\end{definition}

Note that in Definition~\ref{def:new_resolution} $H_i$ can be $H_j$-tight only if $H_i$ is a single vertex with $\deg_{G_{i-1}}(v)=k-2$. 
A natural way to satisfy (TIGHT) condition is to show that whenever there is an $H_j$ such that more than $\beta$ subgraphs $H_{a_1},\ldots,H_{a_\beta},H_{a_{\beta+1}}$ are $H_j$-tight, then 
\[H_j\cup \bigcup_{i\in\{1,\ldots,\beta,\beta+1\}} H_{a_i}\in \mathcal{F}.
\]
If two adjacent vertices have many common neighbors, we get a book, which will be in $\mathcal{F}$.

We are now ready to state and proof our main lemma.

\begin{lemma}[Reducible configurations for weak flexibility]\label{lem:weak}
  For integers $k\ge 4$, $b\ge 1$, $\beta\ge 0$, and for a set $\mathcal{F}$ of forbidden subgraphs, 
  let $G$ be a $\mathcal{F}$-free graph with an enhanced weak $(\mathcal{F},k,b,\beta)$-resolution. 
 Then, there exists an $\varepsilon>0$ such that $G$  is weakly $\varepsilon$-flexible for lists of size $k$.
\end{lemma}

The proof of the lemma is similar the proof of Lemma~\ref{lem:weighted}  in~\cite{choiclemen}. In particular, we explicitly formulate a few arguments in their proof as a separate claim (Claim \ref{claimforlemma} below) that we use in our proof.
We will also need the following Lemma~\ref{lemma:distribWeak}, which is Lemma 12 in \cite{choiclemen}.

Let $G$ be a graph with a weak $(\mathcal{F},k,b)$-resolution $\mathcal{R}$.
Let $\text{AllFix(G)}$ denote the union of all $\text{Fix}(H)$ over all reducible subgraphs $H$ in the resolution $\mathcal{R}$.

\begin{lemma}[Lemma 12 in \cite{choiclemen}]\label{lemma:distribWeak}
Let $b$ be an integer.
Let $G$ be a graph with list assignment $L$ of size $k$ on $V(G)$. %
Suppose $G$ has a weak $(\mathcal{F},k,b)$-resolution, $G$ is $L$-colorable, and there exists a probability distribution on the $L$-colorings $\varphi$ of $G$ such that for every $v\in \text{AllFix}(G)$ and $c\in L(v)$, $\mathbf{Prob}[\varphi(v)=c]\ge\varepsilon$.
Then $G$ with $L$ is weakly $\left(\varepsilon\cdot \frac{1}{b}\right)$-flexible.
\end{lemma}

\noindent
\begin{proof}[Proof of Lemma~\ref{lem:weak}]
For $1 \leq j \leq M+1$, let $\mathcal{H}_j$ be the set of all $H_j$-tight subgraphs where the (TIGHT) property applied.
Let $H_i \in \mathcal{H}_j$ for some $i$ and $j$. 
This means that $H_i$ is a single vertex with $k-2 \geq 2$ neighbors in $H_j$. Hence $\mathcal{H}_i = \emptyset$.

Now, we refactor the enhanced weak $(\mathcal{F},k,b,\beta)$-resolution $\mathcal R$ into an enhanced weak $(\mathcal{F},k,b+\beta,0)$-resolution $\mathcal{R'}$.
To do so, we attach all $H_j$-tight subgraphs to $H_j$ and thus we create a larger configuration $H'_j$. 
The vertices in tight subgraphs are not part of any Fix set.
Formally
\[H'_j\df 
\begin{cases}
\emptyset & \text{ if exists } i \text{ such that }  H_j \in \mathcal{H}_i \\ 
\displaystyle H_j\cup \bigcup_{H \in \mathcal{H}_j} H & \text{ otherwise}
\end{cases}
\]
and $\mathrm{Fix}(H_i') = \mathrm{Fix}(H_i)$ if $H_i' \neq \emptyset$ and $\mathrm{Fix}(H_i') = \emptyset$ otherwise.
Observe that by the (TIGHT) property, the size of the resulting $H'_j$ will be upper-bounded by $b+\beta$ and that $H'_j$ is enhanced weakly $(\mathcal{F},k)$-boundary-reducible or only weakly $(\mathcal{F},k)$-boundary-reducible (provided its neighbourhood is always a loose or small set) if it is not empty.
We simultaneously remember both $\mathcal R$ and $\mathcal R'$, since each time we are using $H'_j$ (or $G'_j$) we are referring to $\mathcal R'$ and each time we are using $H_j$ (or $G_j$) we are referring to $\mathcal R$.

The next step is to create a probability distribution on $L$-colorings $\varphi$ of $G_i$ for all $i$ starting with $G'_{M}$.
Let $p=k^{-{(b+\beta)}}$ and $\varepsilon'=p^{k-1}$.
We are going to show that each $i$ satisfies the following properties:
\begin{enumerate}[(i)]
  \item\label{p:1} for every $v\in \text{AllFix}(G'_i)$ and a color $c\in L(v)$, the probability that $\varphi(v)=c$ is at least $\varepsilon'$, and
  \item\label{p:2} for every color $c$ and every $\mathcal{F}$-free set $I$ in $G'_i$ of size at most $k-3$, the probability that $\varphi(v)\neq c$ for all $v\in I$ is at least $p^{|I|}$.
  \item\label{p:3} for every color $c$ and every loose $\mathcal{F}$-free set $I$ in $G'_i$ of size exactly $k-2$, the probability that $\varphi(v)\neq c$ for all $v\in I$ is at least $p^{|I|}$.
\end{enumerate}

Note that for $G'_{M+1}$ all of the properties trivially hold.
Note that Property~(\ref{p:1}) on $G_0'=G_0=G$ immediately implies that $G$ with $L$ is weakly $\left(\varepsilon'\cdot \frac{1}{b}\right)$-flexible by Lemma~\ref{lemma:distribWeak} and therefore weakly $\varepsilon$-flexible for $\epsilon={\epsilon'\over b}$.

\smallskip 
  
We will make use of the following claim proven implicitly in~\cite{choiclemen}.

\begin{claim}[Implicit in the proof Lemma~13 in~\cite{choiclemen})]\label{cl:implicit}\label{claimforlemma}
  Suppose that we have an enhanced weak $(\mathcal{F},k,b+\beta,0)$-resolution and a probability distribution on $L$-colorings of $G'_{i+1}$ satisfying Properties~(\ref{p:1}),~(\ref{p:2}), and~(\ref{p:3}) on $G'_{i+1}$. %
  If for each vertex $v\in Fix(H'_i)$ and for each $I=N(v)\cap H'_j$ where $j>i$ one of the following holds:
  \begin{itemize}
    \item[(a)] $|I|=k-2$ and $I$ is loose in $H'_j$, or
    \item[(b)] $|I|<k-2$
  \end{itemize}
  then there exists a probability distribution on $L$-colorings of $G'_i$ such that Properties~(\ref{p:1}),~(\ref{p:2}), and~(\ref{p:3}) are satisfied on $G'_{i}$.
\end{claim}

In order to use Claim~\ref{cl:implicit}, we need verify (a) and (b).
If $H_i'$ is not a single vertex $v$ with $\deg_{G_i}(b)=k-2$, then (a) or (b) hold by the definition of $H_i'$.
Hence we need to check the case of $H_i'$ being a single vertex $v$ with $\deg_{G_i}(v)=k-2$. We do it by showing $v$ is not $H_j'$ tight for any $j \geq i$ in the following claim. It implies that for $H_i'$, either (a) or (b) is satisfied. In particular, we will show that we got rid of all tight subgraphs when we refactored $\mathcal{R}$ into $\mathcal{R'}$.

\begin{claim}\label{cl:loose}
There are no $i < j$ such that $H_i'$ is $H_j'$-tight.
\end{claim}
\begin{proof}
Suppose for contradiction that $H_i'$ is $H_j'$-tight for some $i < j$.
By the definition, $H_i'$ is one vertex $v$ with degree $k-2$ in $G_{i-1}$.
By the definition of $\mathcal{R}'$, $v$ is not $H_\ell$-tight for any $\ell > i$.
In particular, $v$ is not $H_j$-tight.
Since $v$ is $H_j'$-tight, $\mathcal{H}_j$ is not empty. 
Hence $H_j'$ is a union of $H_j$ and vertices $W$, where every $w\in W$ has $k-2$ neighbors in $H_j$.
Since $v$ is not $H_j$-tight, it has at most $k-3$ neighbors in $H_j$ and at least one in $W$.
Notice that every vertex $w$ in $W$ has $k-2$ neighbors in $\mathcal{H}_j$ hence a list of $k-1$ colors suffices for extending any coloring of $\mathcal{H}_j$ to $w$ greedily.
This and the (FORB) property for $H_j$ imply that $v$ is not $H_j'$-tight because $N(v)$ in $H_j'$ is loose, which is a contradiction.
\end{proof}

We conclude that Claim~\ref{cl:loose} enables us to use Claim~\ref{cl:implicit} directly on $\mathcal{R'}$.
This finishes the proof of Lemma~\ref{lem:weak}.
\end{proof}
\medskip

For a positive integer $d$, a {\it $d$-vertex}, a {\it $d^+$-vertex}, and a {\it $d^-$-vertex} are a vertex of degree $d$, at least $d$, and at most $d$, respectively. 
A {\it $d$-face}, a {\it $d^+$-face}, and a {\it $d^-$-face} are defined analogously. 
A {\it $(d_1, d_2, d_3)$-face} is a $3$-face where the degrees of the vertices on the face are $d_1, d_2, d_3$.  
We will sometimes call 3-faces {\em triangles}. 
A \emph{diamond} $D$ is a graph isomorphic to $K_4$ minus an edge. 
The $2$-vertices of $D$ are be called the {\em side vertices}, and the  $3$-vertices will be called the {\em middle vertices} of $D$. 
For a vertex $v$, denote by $d(v)$ the degree of $v$.
Let $G$ be a graph. By $T(a,b,c)$ we denote a triangle in $G$ vertices of degree $a$, $b$, and $c$ in $G$, and by $Dia(a-b,c,d)$ a diamond in $G$ with middle vertices of degrees $a$ and $b$ and side vertices of degrees $c$ and $d$.

\begin{lemma}\label{deglem}
Let $G$ be a plane $\{C_5, C_6, C_7\}$-free graph. Suppose that $v$ is the middle vertex of $k$ distinct diamonds, and $v$ is adjacent to $m$ faces of size 3 or 4 that are not part of a diamond in which $v$ is a middle vertex. Then $d(v) \ge 3k + 2m$ and $k\le  \lfloor \frac{d(v)}{3}\rfloor$. 
\end{lemma}
\begin{proof}
If not, then $v$ is adjacent to three faces $f, g, h$, each of them of size at most 4, such that $f,g$ share an edge and $g,h$ share an edge. But this induces a cycle $C_i$ with  $5\le i \le 7$, a contradiction.
\end{proof}
\medskip

In all the figures in the paper, black vertices have all their incident edges drawn, whereas a white vertex may have more edges incident than drawn (since white vertices are in the boundary).

\section{Proof of Theorem \ref{main567}}

\subsection{Reducible Configurations}

Let $\mathcal{F}=\{K_4,C_5,C_6,C_7,B_5\}$.
In this section we will provide a handful of $(\mathcal{F},4)$-boundary-reducible configurations. 

\begin{lemma}\label{reduciblemain567}
The following configurations are  $(\mathcal{F},4)$-boundary-reducible. See Figure~\ref{fig:Cconfs} for reference. 
If boundary is not mentioned, it is empty.
\begin{itemize}
    \item[{\em (C1)}] A vertex of degree at most 2.
    \item[{\em (C2)}] Three $3$-vertices  appearing on a path of length 2. 
    \item[{\em (C3)}] The triangle $T(3,3,3)$. 
    \item[{\em (C4)}] Let $u$ be a $3$-vertex adjacent to the middle $4$-vertex of the diamond $D = Dia(4-3,4,5^+)$. Let $v$ denote the $5^+$-vertex that is a side vertex of $D$. 
    Then $D \cup \{u\}$ is reducible with boundary $v$.
    \item[{\em (C5)}] $Dia(3 - 3, 5^+,5^+)$ with $5^+$-vertices in the boundary.
    \item[{\em (C6)}] $Dia(3 - 5^+,3,5^+)$ with $5^+$-vertices in the boundary.
    \item[{\em (C7)}] The diamond $Dia(5-4,3,3)$. 
    \item[{\em (C8)}] Let $D_1 = Dia(4-4,5,3)$ and $D_2 = (5-3,4,4^+)$ be two diamonds sharing the same $5$-vertex. Let $v$ denote the $4^+$-vertex that is a side vertex of $D_2$. Then the subgraph $D_1 \cup D_2$ is reducible with $v$ in the boundary. 
    \item[{\em (C9)}] Let $D_1 = D_2 = Dia(3-4,4,5^+)$ be two diamonds whose middle $4$-vertices are connected by an edge. Let $v_1$ and $v_2$ denote the two $5^+$-vertices that are the side vertices of $D_1$ and $D_2$. Then the subgraph $D_1 \cup D_2$ is reducible with $v_1$ and $v_2$ in the boundary. 
    \item[{\em (C10)}] Let $D_1 = Dia(4-3,5^+,5)$ and $D_2 = Dia(5-3,4,4^+)$ be two diamonds sharing a middle $5$-vertex. Let $v_1$ denote the $5^+$-vertex that is a side vertex of $D_1$ and let $v_2$ denote the $4^+$-vertex that is a side vertex of $D_2$. Then the subgraph $D_1 \cup D_2$ is reducible with $v_1$ and $v_2$ in the boundary.
    \item[{\em (C11)}] A diamond $Dia(4-4,3,4)$ along with a $3$-vertex adjacent to one of the middle $4$-vertices. 
    \item[{\em (C12)}] Let $D_1 = Dia(3-4,4,5^+)$ and $D_2 = Dia(4-4,4,3)$ be two diamonds whose middle $4$-vertices are connected by an edge. Let $v$ denote the $5^+$-vertex that is a side vertex of $D_1$. Then the subgraph $D_1 \cup D_2$ is reducible is reducible with $v$ in the boundary.
    \item[{\em (C13)}] Let $D_1 = Dia(3-5,5,5)$ and $D_2,D_3 = Dia(5-3,4,5+)$ be two diamonds where the two side $5$-vertices of $D_1$ are middle vertices of $D_2$ and $D_3$. Let $v_1$ and $v_2$ denote the two side $5^+$-vertices of $D_2$ and $D_3$, respectively. Then $D_1 \cup D_2 \cup D_3$ is reducible with $v_1$ and $v_2$ in the boundary. 
\end{itemize}
\end{lemma}

We wish to point out that the reducible configurations are meant to be induced subgraphs by definition, and we will use them as such in the discharging part of the proof. 
The only configuration, where two external edges can be identified is (C2) and it gives (C3), which we explicitly list. 
It can be straightforwardly checked that no identification of vertices in (C1)--(C15) is possible since otherwise, it creates a forbidden subgraph.
\medskip

\begin{proof}[Proof of Lemma~\ref{reduciblemain567}]
  It is straightforward to check that each configuration (C1)--(C15) satisfies the (FIX) and (FORB) conditions in Definition~\ref{reduce}. However, checking all the cases is rather tedious. Hence we developed a simple computer program that does it, see \oururl\footnote{This program is also available as a part of the sources in our arXiv submission.}. In particular, a greedy coloring works in all cases. For an interested reader who wishes to check some cases by hand, we added list sizes to Figure~\ref{fig:Cconfs}. We also provide here proofs showing that (C2) and (C5) are $(\mathcal{F},4)$-boundary-reducible. Together, these two configurations demonstrate how to prove that the remaining configurations are reducible. 

The two reducible configurations $H_1$ and $H_2$ corresponding to (C2) and (C5), respectively are depicted in Figure~\ref{config:example:C2:C5}. The reduced parts $R_1 = H_1$ and $R_2 \subset H_2$ are provided as well. Finally, we have labeled each vertex in the figure with the value of the function $4 - \text{deg}_{H_i} + \text{deg}_{R_i}$ for $i \in \{1,2\}$. 
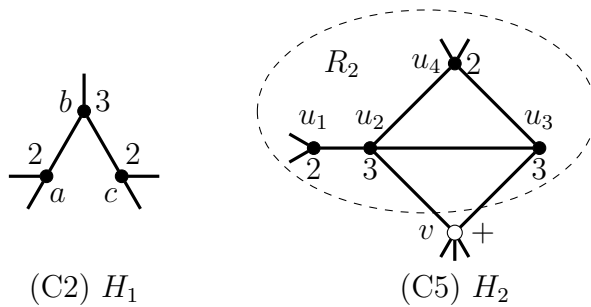
\begin{figure}[H]
    \centering
    \begin{tikzpicture}
    \node[vtx] (w) at (1,0) {};
    \path (w) ++(60:1) node (u)[vtx] {};
    \path (w) ++(0:1) node (v)[vtx] {};

    \draw[gedge] (u) -- (v);
    \draw[gedge] (w) -- (u);

    \draw[gedge] (u) -- ++(90:0.5);
    \draw[gedge] (v) -- ++(300:0.5);
    \draw[gedge] (w) -- ++(240:0.5);
    \draw[gedge] (w) -- ++(180:0.5);
    \draw[gedge] (v) -- ++(0:0.5);

    \node at (1.5,-1.5) {(C2) $H_1$};

    \node at (0.85,0.3) {$2$};
    \node at (2.15,0.3) {$2$};
    \node at (1.75,1) {$3$};
    \node at (1.15,-0.25) {$a$};
    \node at (1.85,-0.25) {$c$};
    \node at (1.25,1) {$b$};
\end{tikzpicture} \hspace{1 cm}
\begin{tikzpicture}[scale=0.75]
\baseConfigurationdiamond{}{}{}{}

\node at (-2,1.5) {$R_2$};

\node[vtx] (z) at (-2.5,0) {};

\draw[gedge] (u) -- (z);

\draw[gedge] (z)--++(150:0.5);
\draw[gedge] (z)--++(210:0.5);

\draw[gedge] (x)--++(120:0.5);
\draw[gedge] (x)--++(60:0.5);

\draw[gedge] (y)--++(270:0.5);
\draw[gedge] (y)--++(300:0.5);
\draw[gedge] (y)--++(240:0.5);

\draw[dashed] (-0.5,0.65) ellipse[x radius = -3 cm, y radius = 1.8cm];

\node[bndry] at (y) {};
\node at (0.5,-1.5) {$+$}; 
\node at (-2.5,0.5) {$u_1$};
\node at (-0.5,-1.5) {$v$};
\node at (-1.5,0.5) {$u_2$};
\node at (1.5,0.5) {$u_3$};
\node at (-0.5,1.5) {$u_4$};
\draw (z) node[below]{2};
\draw (x) node[right]{2};
\draw (u) node[below]{3};
\draw (v) node[below]{3};
\node at (0,-2.5) {(C5) $H_2$};
\end{tikzpicture}
    \caption{Reducible configurations (C2) and (C5). The reduced parts consist of the black vertices. }
    \label{config:example:C2:C5}
\end{figure}
By definition, checking the (FIX) condition for any subgraph $H$ with reducible part $R$ is equivalent to showing that for each $v\in V(R)$, $R$ can be properly colored after assigning each vertex a list of size $((4-\text{deg}_H+\text{deg}_R)\downarrow v)$. It is clear by inspection that this is the case for $\text{(C2)} = H_1 = R_1$, and hence we only need to check the (FORB) condition for (C2). Since (FIX) is already verified, it implies (FORB) for subsets of size one in $R$.
It remains to verify (FORB) for subsets of size two in $R$.

If we apply (FORB) to $a$ and $c$, then both $a$ and $b$ will be left with one available color in their lists. Vertex $b$ still has three colors in it's list.
Therefore, we can greedily color $a$, $c$, and $b$ in this order to obtain a proper coloring for (C2).
If we apply (FORB) to $a$ and $b$, then the color for $a$ will be fixed, and each of $b$ and $c$ will be left with two possible colors. Therefore, we can greedily color $a$, $b$, and $c$ in this order to obtain a proper coloring for (C2). By symmetry, the case of applying (FORB) to $b$ and $c$ is also verified, implying that (C2) is reducible. 

Let $H_2$ be a subgraph of $G$ isomorphic to (C5). Let $R_2 \subset H_2$ denote the reducible part of (C5), i.e. the subgraph of $H_2$ induced by vertices $u_1,\dots,u_4$.
For each $i = 1,\dots,4$, we will check the (FIX) condition for $u_i$. Let $L_i$ be an arbitrary list assignment where each vertex in $R$ is assigned a list of size $((4-\text{deg}_H+\text{deg}_R)\downarrow u_i)$. We will now show that $R$ can be properly colored.
In each case we list the order of vertices in greedy coloring.
\begin{itemize}
    \item $L_1:$ $u_1,u_2,u_4,u_3$.
    \item $L_2:$ $u_2,u_1,u_4,u_3$. 
    \item $L_3:$ $u_3,u_4,u_2,u_1$. 
    \item $L_4:$ $u_4,u_3,u_2,u_1$. 
\end{itemize}

Next we need to verify that $H$ satisfies the (FORB) condition.
However, only one subset of $R$ of size two is $\mathcal{F}$-free: $\{u_1,u_2\}$. 
In that case $R$ can be colored greedily in the following order $u_1,u_2,u_4,u_3$.
Thus, (C5) is a reducible configuration. 
\end{proof}

\newpage

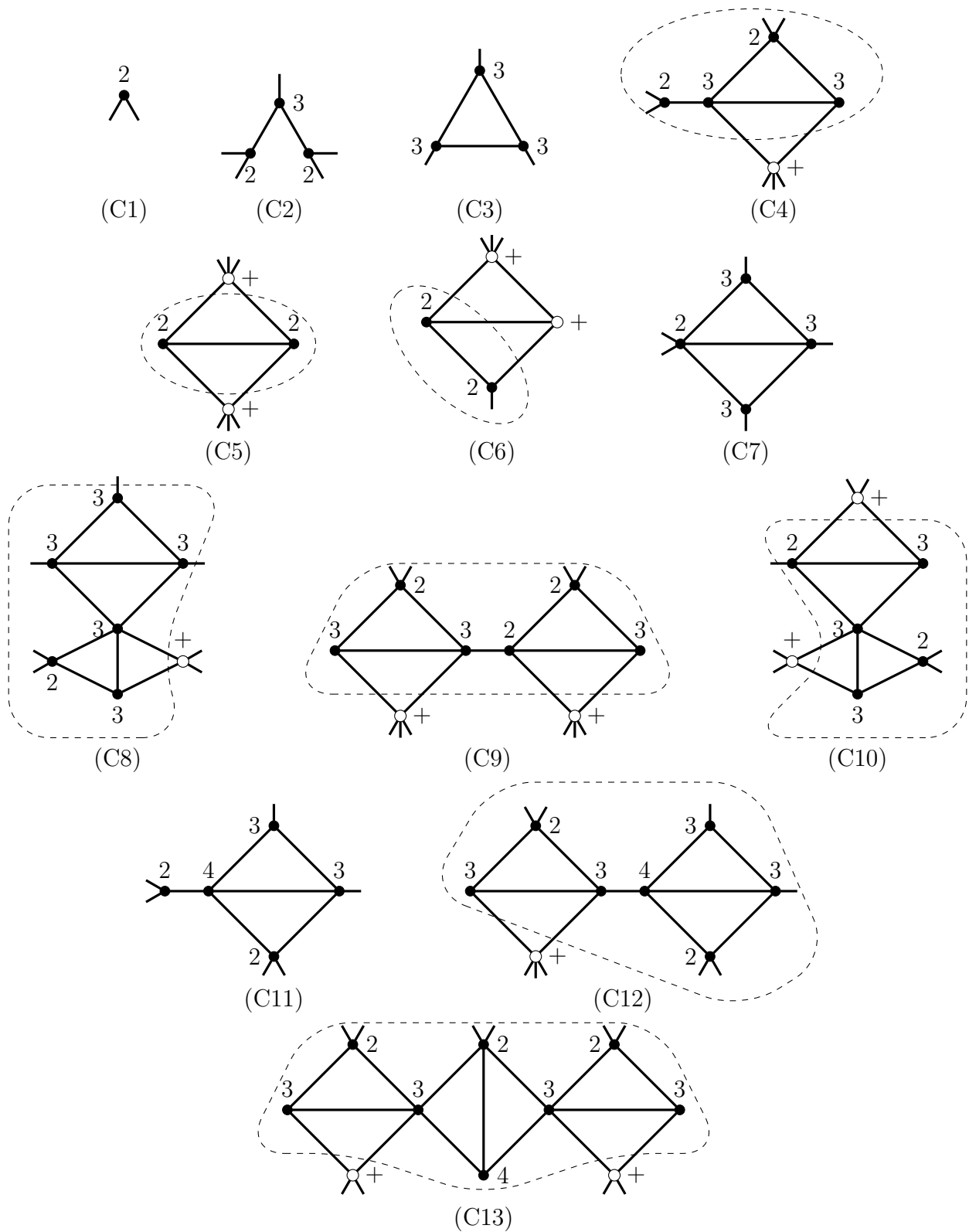
\begin{figure}[H]
    \centering
\begin{tikzpicture}
\node[vtx,label=$2$] (v) at (0,1) {};
\draw[gedge] (v) -- ++(240:0.5) {};
\draw[gedge] (v) -- ++(300:0.5) {};

\node at (0,-1) {(C1)};
\end{tikzpicture}
\hspace{ 2 em } 
\begin{tikzpicture}
\node[vtx,label=below:$2$] (w) at (1,0) {};
\path (w) ++(60:1) node (u)[vtx,label=right:$3$] {};
\path (w) ++(0:1) node (v)[vtx,label=below:$2$] {};

\draw[gedge] (u) -- (v);
\draw[gedge] (w) -- (u);

\draw[gedge] (u) -- ++(90:0.5);
\draw[gedge] (v) -- ++(300:0.5);
\draw[gedge] (w) -- ++(240:0.5);
\draw[gedge] (w) -- ++(180:0.5);
\draw[gedge] (v) -- ++(0:0.5); 

\node at (1.5,-1) {(C2)};
\end{tikzpicture}
\hspace{ 2 em }
\begin{tikzpicture}[scale=0.75]
\baseConfigurationtriangle{$3$}{$3$}{$3$}

\draw[gedge] (u) -- ++(90:0.5);

\draw[gedge] (v) -- ++(-60:0.5);

\draw[gedge] (w) -- ++(240:0.5);
\node at (2,-1.5) {(C3)};
\end{tikzpicture} \hspace{ 2 em} 
\begin{tikzpicture}[scale=0.75]
\baseConfigurationdiamond{$3$}{$3$}{$2$}{}

\node[vtx,label=above:$2$] (z) at (-2.5,0) {};

\draw[gedge] (u) -- (z);

\draw[gedge] (z)--++(150:0.5);
\draw[gedge] (z)--++(210:0.5);

\draw[gedge] (x)--++(120:0.5);
\draw[gedge] (x)--++(60:0.5);

\draw[gedge] (y)--++(270:0.5);
\draw[gedge] (y)--++(300:0.5);
\draw[gedge] (y)--++(240:0.5);

\draw[dashed] (-0.5,0.65) ellipse[x radius = -3 cm, y radius = 1.5cm];

\node[bndry] at (y) {};
\node at (0.5,-1.5) {$+$}; 
\node at (0,-2.5) {(C4)}; 
\end{tikzpicture}
\hspace{2 em}
\begin{tikzpicture}[scale = 0.75]
\baseConfigurationdiamond{$2$}{$2$}{}{}
\draw[gedge] (x)--++(120:0.5);
\draw[gedge] (x)--++(60:0.5);
\draw[gedge] (x)--++(90:0.5);

\draw[gedge] (y)--++(270:0.5);
\draw[gedge] (y)--++(300:0.5);
\draw[gedge] (y)--++(240:0.5);

\draw[dashed] (0,0) ellipse[x radius = 2 cm, y radius = 1.15cm];
\node at (0.5,-1.5) {$+$}; 
\node at (0.5,1.5) {$+$}; 
\node[bndry] at (y) {};
\node[bndry] at (x) {};
\node at (0,-2.5) {(C5)};
\end{tikzpicture}
\hspace{2 em}
\begin{tikzpicture}[scale = 0.75]
\baseConfigurationdiamond{$2$}{}{}{$2$}
\draw[gedge] (x)--++(120:0.5);
\draw[gedge] (x)--++(60:0.5);
\draw[gedge] (x)--++(90:0.5);

\draw[gedge] (y)--++(270:0.5);

\draw[dashed] (-0.75,-0.75) ellipse[x radius = 2 cm, y radius = 1 cm, rotate = 135];
\node at (0.5,1.5) {$+$}; 
\node at (2,0) {$+$}; 
\node[bndry] at (x) {};
\node[bndry] at (v) {};
\node at (0,-3) {(C6)};
\end{tikzpicture}
\hspace{2 em}
\begin{tikzpicture}[scale = 0.75]
\baseConfigurationdiamond{$2$}{$3$}{$3$}{$3$}
\draw[gedge] (y)--++(270:0.5);
\draw[gedge] (x)--++(90:0.5);
\draw[gedge] (u)--++(150:0.5);
\draw[gedge] (u)--++(210:0.5);
\draw[gedge] (v)--++(0:0.5);

\node at (0,-2.5) {(C7)};
\end{tikzpicture}
\hspace{2 em}
\begin{tikzpicture}[scale = 0.75]
\baseConfigurationdiamond{$3$}{$3$}{$3$}{$3$}

\node[vtx,label = below:$3$] (a) at (0,-3) {};
\node[vtx,label= below:$2$] (b) at (-1.5,-2.25) {};
\node[vtx,label=$+$] (c) at (1.5,-2.25) {};

\draw[gedge] (a)--(y);
\draw[gedge] (b)--(y);
\draw[gedge] (c)--(y);
\draw[gedge] (b)--(a);
\draw[gedge] (c)--(a);

\draw[gedge] (b)--++(150:0.5);
\draw[gedge] (b)--++(210:0.5);
\draw[gedge] (c)--++(30:0.5);
\draw[gedge] (c)--++(-30:0.5);

\draw[gedge] (x)--++(90:0.5);
\draw[gedge] (u)--++(180:0.5);
\draw[gedge] (v)--++(0:0.5);

\draw[dashed, rounded corners=20pt] (1.5,-4)  -- (-2.5,-4) -- (-2.5,1.8) --  (2.5,1.8)  --  (1,-2)  --   cycle;
\node[bndry] at (c) {};
\node at (0,-4.5) {(C8)};
\end{tikzpicture}
\hspace{2 em}
\begin{tikzpicture}[scale = 0.75]
\baseConfigurationdiamond{$2$}{$3$}{$2$}{}

\node[vtx,label=above:$3$] (a) at (-2.5,0) {};
\node[vtx,label=above:$3$] (b) at (-5.5,0) {};
\node[vtx,label=right:$2$] (c) at (-4,1.5) {};
\node[vtx] (d) at (-4,-1.5) {};
\draw [gedge] (a) -- (b)  (b) -- (c)  (b) -- (d)  (a) -- (d) (a) -- (c) (a)--(u);

\draw[gedge] (d)--++(240:0.5);
\draw[gedge] (d)--++(270:0.5);
\draw[gedge] (d)--++(300:0.5);
\draw[gedge] (y)--++(240:0.5);
\draw[gedge] (y)--++(270:0.5);
\draw[gedge] (y)--++(300:0.5);
\draw[gedge] (x)--++(60:0.5);
\draw[gedge] (x)--++(120:0.5);
\draw[gedge] (c)--++(60:0.5);
\draw[gedge] (c)--++(120:0.5);

\draw[dashed, rounded corners=20pt] (-6.5,-1) --  (2.5,-1) -- ( 1,2) -- ( -5,2) -- cycle;
\node at (0.5,-1.5) {$+$};
\node at (-3.5,-1.5) {$+$};
\node[bndry] at (d) {};
\node[bndry] at (y) {};
\node at (-2,-2.5) {(C9)};
\end{tikzpicture}  
\hspace{2 em}
\begin{tikzpicture}[scale = 0.75]
\baseConfigurationdiamond{$2$}{$3$}{}{$3$}

\node[vtx,label = below:$3$] (a) at (0,-3) {};
\node[vtx] (b) at (-1.5,-2.25) {};
\node[vtx,label = above:$2$] (c) at (1.5,-2.25) {};

\draw[gedge] (a)--(y);
\draw[gedge] (b)--(y);
\draw[gedge] (c)--(y);
\draw[gedge] (b)--(a);
\draw[gedge] (c)--(a);

\draw[gedge] (b)--++(150:0.5);
\draw[gedge] (b)--++(210:0.5);
\draw[gedge] (c)--++(30:0.5);
\draw[gedge] (c)--++(-30:0.5);

\draw[gedge] (x)--++(60:0.5);
\draw[gedge] (x)--++(120:0.5);
\draw[gedge] (u)--++(180:0.5);

\draw[dashed, rounded corners=20pt] (2.5,-4)  -- (-2.5,-4) -- (-0.5,-2) -- (-2.5,1) --  (2.5,1)  -- cycle;
\node at (-1.5,-1.75) {$+$};
\node at (0.5,1.5) {$+$};

\node[bndry] at (b) {};
\node[bndry] at (x) {};
\node at (0,-4.5) {(C10)};
\end{tikzpicture} \hspace{ 2 em} 
\begin{tikzpicture}[scale=0.75]
\baseConfigurationdiamond{$4$}{$3$}{$3$}{$2$}

\node[vtx,label = above:$2$] (z) at (-2.5,0) {};

\draw[gedge] (u) -- (z);

\draw[gedge] (z)--++(150:0.5);
\draw[gedge] (z)--++(210:0.5);

\draw[gedge] (x)--++(90:0.5);

\draw[gedge] (v)--++(0:0.5);

\draw[gedge] (y)--++(300:0.5);
\draw[gedge] (y)--++(240:0.5);
\node at (0,-2.5) {(C11)}; 
\end{tikzpicture}
\hspace{2 em}
\begin{tikzpicture}[scale = 0.75]
\baseConfigurationdiamond{$4$}{$3$}{$3$}{$2$}

\node[vtx,label=above:$3$] (a) at (-2.5,0) {};
\node[vtx,label = above:$3$] (b) at (-5.5,0) {};
\node[vtx,label = right:$2$] (c) at (-4,1.5) {};
\node[vtx] (d) at (-4,-1.5) {};
\draw [gedge] (a) -- (b)  (b) -- (c)  (b) -- (d)  (a) -- (d) (a) -- (c) (a)--(u);

\draw[gedge] (d)--++(240:0.5);
\draw[gedge] (d)--++(270:0.5);
\draw[gedge] (d)--++(300:0.5);
\draw[gedge] (y)--++(240:0.5);
\draw[gedge] (y)--++(300:0.5);
\draw[gedge] (x)--++(90:0.5);
\draw[gedge] (c)--++(60:0.5);
\draw[gedge] (c)--++(120:0.5);
\draw[gedge] (v)--++(0:0.5);

\draw[dashed, rounded corners=20pt] (-6.5,0) -- (0.75,-2.75) --  (2.75,-1.5) -- ( 1,2.5) -- ( -5,2.5) -- cycle;
\node at (-3.5,-1.5) {$+$};
\node[bndry] at (d) {};
\node at (-2,-2.5) {(C12)};
\end{tikzpicture}
\hspace{2 em}
\begin{tikzpicture}[scale = 0.75]
\baseConfigurationdiamond{$3$}{$3$}{$2$}{}

\node[vtx,label = above:$3$] (a) at (-4.5,0) {};
\node[vtx,label = above:$3$] (b) at (-7.5,0) {};
\node[vtx,label = right:$2$] (c) at (-6,1.5) {};
\node[vtx] (d) at (-6,-1.5) {};
\draw [gedge] (a) -- (b)  (b) -- (c)  (b) -- (d)  (a) -- (d) (a) -- (c);

\node[vtx,label = right:$2$] (k) at (-3,1.5) {};
\node[vtx,label = right:$4$] (j) at (-3,-1.5) {};
\draw [gedge] (k) -- (a)  (j) -- (a)  (k) -- (u)  (j) -- (u) (k) -- (j);

\draw[gedge] (k)--++(60:0.5) (k)--++(120:0.5);
\draw[gedge] (d)--++(240:0.5);
\draw[gedge] (d)--++(300:0.5);
\draw[gedge] (y)--++(240:0.5);
\draw[gedge] (y)--++(300:0.5);
\draw[gedge] (x)--++(60:0.5);
\draw[gedge] (x)--++(120:0.5);
\draw[gedge] (c)--++(60:0.5);
\draw[gedge] (c)--++(120:0.5);

\draw[dashed, rounded corners=20pt] (-8.5,-1) -- (-6,-1) -- (-3,-2) -- (0,-1)-- (2.5,-1) -- ( 1,2) -- (-7,2) -- cycle;
\node at (0.5,-1.5) {$+$};
\node at (-5.5,-1.5) {$+$};
\node[bndry] at (d) {};
\node[bndry] at (y) {};
\node at (-3,-2.5) {(C13)};
\end{tikzpicture}  
    \caption{Reducible Configurations for Theorem~\ref{main567}. The labels give the list sizes remaining after accounting for the external neighbors and boundary vertices.}
    \label{fig:Cconfs}
\end{figure}

\subsection{Discharging}
In this section we prove the following lemma, which by Lemma \ref{lem:weighted} implies Theorem \ref{main567}.
\begin{lemma}\label{lem:unavoidable B5}
Let $G$ be a connected $\{K_4,C_5,C_6,C_7,B_5\}$-free plane graph.
Then $G$ contains at least one of the reducible configurations \emph{\textrm{(C1)--(C13)}}.
\end{lemma}

\begin{proof} Suppose for contradiction that $G$ is a  connected $\{K_4,C_5,C_6,C_7,B_5\}$-free plane graph that contains none of the configurations (C1)--(C13). 
We use discharging to obtain a contradiction with Euler's formula.

We denote the initial charge by $ch$.
For every vertex $v$, we let $ch(v) = \deg(v)-4$, and every face $f$ we let $ch(f) = \ell(f)-4$, where $\ell(F)$ is the length of the facial walk. 
For convenience we will also assign charge to the edges of $G$. The initial charge is 0 for each edge.
By Euler's formula, the total sum of initial charges is $-8$.

We sequentially apply the following rules that move the charge around, while keeping the sum of charges unchanged. The charge at the end is called the \emph{final charge}.
The final charges will be all nonnegative, contradicting that their sum is $-8$.
\begin{itemize}
\item[(R1)] Every $8^+$-face sends charge $1/2$ to every incident $3$-face and $4$-face for every edge they have in common. 
\item[(R2)] 
For every edge $e$ that is not incident with any $3$-face or $4$-face the following applies.
If $e$ is a bridge, $e$ receives charge $1$ from the unique face incident with $e$. If $e$ is not a bridge, $e$ receives charge $1/2$ for each of the two faces incident to $e$.
\item[(R3)] For every vertex $u$ and an incident edge $e = uz$ with charge $1$:
\begin{itemize}
    \item[(R3a)] If $u$ and $z$ are both 3-vertices, then $e$ sends charge $1/2$ to $u$. 
    
    \item[(R3b)] If $u$ is a 3-vertex and $z$ is $4^+$-vertex, then $e$ sends charge $1$ to $u$. 

    \item[(R3c)] If $z$ is a $4^+$-vertex, $u$ is the middle $4$-vertex of the diamond $Dia(4-3,4,4^+)$, and $v$ is the $3$-vertex on this diamond, then $e$ sends $1$ to $v$. 
    
    \item[(R3d)] If $z$ is a $4^+$-vertex,  $u$ is one of the middle $4$-vertices of the diamond $Dia(4-4,3,4)$, and $v$ is the $3$-vertex on this diamond, then $e$ sends $1/2$ to $v$.
\end{itemize}

\item[(R4)] Every $4$-face sends charge $1$ to each incident  $3$-vertex.

\item[(R5)] Let $f$ be a $3$-face that is not part of a diamond. If exactly one vertex $u$ of $f$ has degree $3$, then $f$ sends  $1/2$ to $u$. 

\item[(R6)] The following rules apply for a $5$-vertex $u$.
If $u$ is a middle vertex in
\begin{itemize}
    \item[(R6a)]  $Dia(5-3,4,4)$ or $Dia(5-3,5,4)$, then $u$ sends $1$ to the middle $3$-vertex;
    
    \item[(R6b)]  $Dia(5-3,6^+,4^+)$  or $Dia(5-3,5,5)$, then $u$ sends $1/2$ to the middle $3$-vertex; 

    \item[(R6c)] $Dia(5-5^+,3,3)$, then $u$ sends $1/4$ to each of the two side $3$-vertices;  
    
    \item[(R6d)] $Dia(5-4^+,4^+,3)$, then $u$ sends $1/2$ to the side $3$-vertex. 
\end{itemize}

\item[(R7)] The following rules apply for every $5$-vertex $u$ and a diamond $D$, where $v$ is a side vertex of $D$. If $D$ is
\begin{itemize}
    \item[(R7a)] $Dia(4-4,5,3)$, then $u$ sends $1/2$ to the side $3$-vertex;
    
    \item[(R7b)] $Dia(4-3,5,5^+)$, then $u$ sends $1/2$ to the middle $3$-vertex;
    
    \item[(R7c)] $Dia(5-3,5,4^+)$ and $u$ has not already sent $1$ to another diamond under rule (R6a), then $u$ sends $1/2$ to the middle $3$-vertex. 
\end{itemize} 

\item[(R8)] 
The following rules apply for every $6^+$-vertex $u$ and a diamond $D$, where $u$ is a side vertex of $D$. If $D$ is
\begin{itemize} 
    \item[(R8a)]  $Dia(5-3,6^+,4^+)$, then $u$ sends $1/2$ to the middle $3$-vertex;
    
    \item[(R8b)]  $Dia(4-4,6^+,3)$, then $u$ sends $1/2$ to the side $3$-vertex;   
    
    \item[(R8c)] $Dia(4-3,6^+,4^+)$, then $u$ sends $1/2$ to the middle $3$-vertex. 
\end{itemize}

\item[(R9)] 
The following rules apply for every $6^+$-vertex $u$ and a diamond $D$, where $u$ is a middle vertex of $D$. If $D$ is
\begin{itemize}
    \item[(R9a)] $Dia(6^+-4^+,3,3)$ then $u$ sends $1/2$ to each of the side vertices;
        \item[(R9b)] $Dia(6^+-4^+,4^+,3)$, then $u$ sends $1$ to the side $3$-vertex; 
    \item[(R9c)]  $Dia(6^+-3,4^+,4^+)$ then $u$ sends $1$ to the other middle $3$-vertex.
\end{itemize}

\end{itemize}

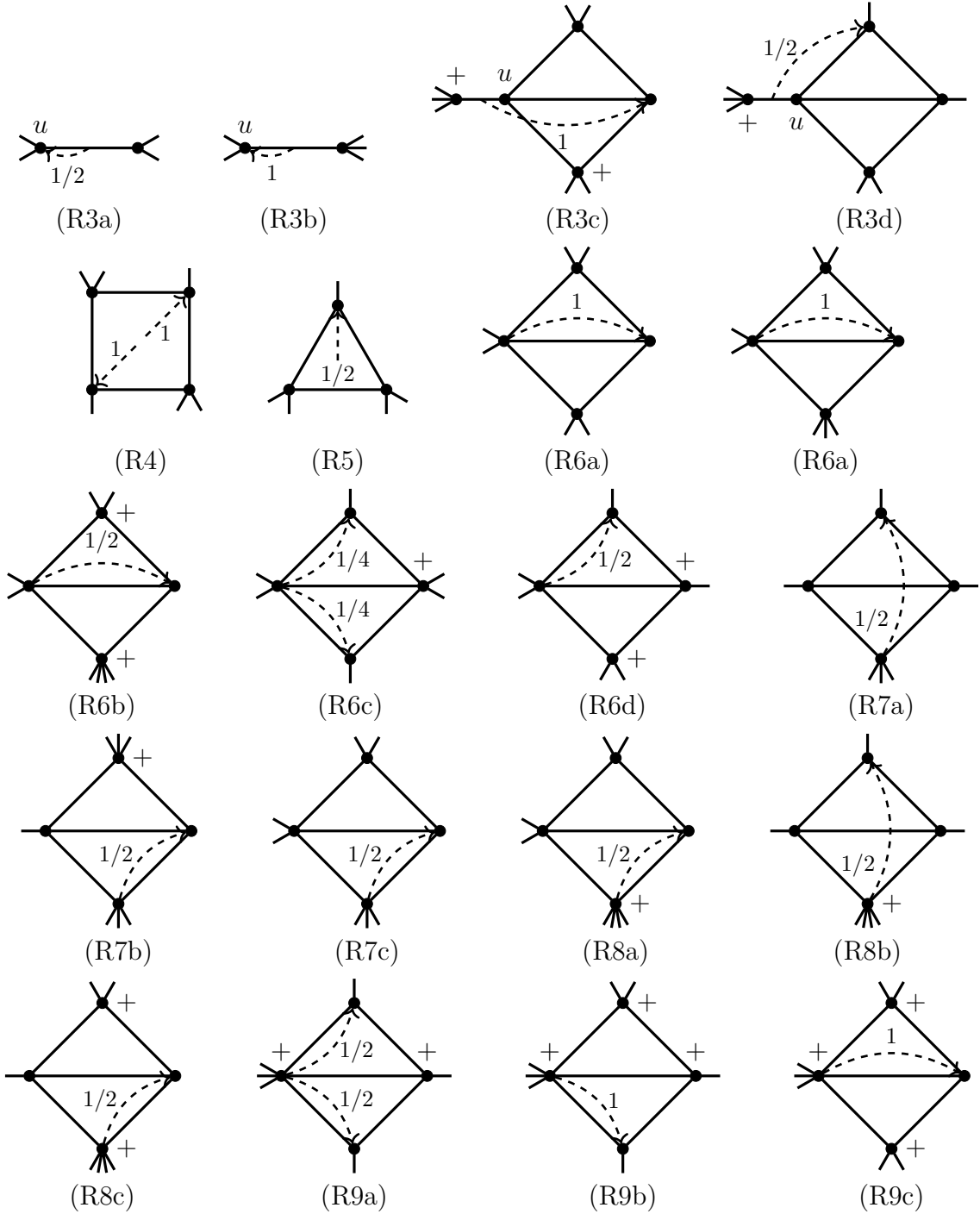
\begin{figure}[H]
    \centering
    \begin{tikzpicture}[scale = 0.75]
    \node[vtx, label = $u$] (u) at (-1,0) {};
    \node[vtx] (z) at (1,0) {}; 
    \draw[gedge] (u) to (z); 
    \draw[gedge] (u)--++(150:0.5);
    \draw[gedge] (u)--++(210:0.5);
    
    \draw[gedge] (z)--++(30:0.5);
    \draw[gedge] (z)--++(-30:0.5);
    
    \path[->]
    \chargepath{(0,0)}{dashedge,bend left}{below}{$1/2$}{(u)};
    
    \node at (0,-1.5) {(R3a)};
    \end{tikzpicture}\hspace{2 em}
    \begin{tikzpicture}[scale = 0.75]
    \node[vtx, label = $u$] (u) at (-1,0) {};
    \node[vtx,] (z) at (1,0) {}; 
    
    \draw[gedge] (u) to (z); 
    \draw[gedge] (u)--++(150:0.5);
    \draw[gedge] (u)--++(210:0.5);
    
    \draw[gedge] (z)--++(30:0.5);
    \draw[gedge] (z)--++(-30:0.5);
    \draw[gedge] (z)--++(0:0.5);
    
    \path[->]
    \chargepath{(0,0)}{dashedge,bend left}{below}{$1$}{(u)};
    
    \node at (0,-1.5) {(R3b)};
    \end{tikzpicture}\hspace{2 em}
    \begin{tikzpicture}[scale = 0.75]
    \baseConfigurationdiamond{}{}{}{}

    \node[vtx, label = $+$] (z) at (-2.5,0) {};
    \draw[gedge] (u)--(z); 
    \draw[gedge] (z)--++(180:0.5);
    \draw[gedge] (z)--++(150:0.5);
    \draw[gedge](z)--++(210:0.5);
    \draw[gedge] (x)--++(120:0.5) (x)--++(60:0.5) (y)--++(240:0.5) (y)--++(300:0.5);
    
    \path[->]
    \chargepath{(-2,0)}{dashedge,bend right}{below}{$1$}{(v)};
    
    \node at (0.5,-1.5) {$+$};
    \node at (-1.5,0.5) {$u$};
    \node at (0,-2.5) {(R3c)};
    \end{tikzpicture}\hspace{2 em}
    \begin{tikzpicture}[scale = 0.75]
    \confgurationD
    
    \node[vtx] (z) at (-2.5,0) {};
    \draw[gedge] (u)--(z); 
    \draw[gedge] (z)--++(180:0.5);
    \draw[gedge] (z)--++(150:0.5);
    \draw[gedge](z)--++(210:0.5);
    
    \path[->]
    \chargepath{(-2,0)}{dashedge, bend left}{left}{$1/2$}{(x)};
    
    \node at (-2.5,-0.5) {$+$};
    \node at (-1.5,-0.5) {$u$};
    \node at (0,-2.5) {(R3d)};
    \end{tikzpicture}\hspace{2 em}
    \begin{tikzpicture}[scale = 0.75]
    \baseConfigurationsquare
    
    \draw[gedge] (x)--++(90:0.5);
    
    \draw[gedge] (v) --++(270:0.5);
    
    \draw[gedge] (u) --++(60:0.5);
    \draw[gedge] (u) --++(120:0.5);
    
    \draw[gedge] (y) --++(240:0.5);
    \draw[gedge] (y) --++(300:0.5);
    
    \path[->]
    \chargepath{(0.1,0.1)}{dashedge}{below}{$1$}{(x)}
    \chargepath{(-0.1,-0.1)}{dashedge}{above}{$1$}{(v);}
    
    \node at (0,-2.5) {(R4)};
    \end{tikzpicture}\hspace{2 em}
    \begin{tikzpicture}[scale = 0.75]
    \baseConfigurationtriangle{}{}{}
    
    \draw[gedge] (u) -- ++(90:0.5);

    \draw[gedge] (v) -- ++(270:0.5);
    \draw[gedge] (v) -- ++(330:0.5);
    
    \draw[gedge] (w) -- ++(210:0.5);
    \draw[gedge] (w) -- ++(270:0.5);
    
    \path[->]
    \chargepath{(2,0.6)}{dashedge}{near start}{}{(u)};
    
    \node at (2,0.3){\footnotesize{1/2}};
    \node at (2,-1.5) {(R5)};
    \end{tikzpicture} \hspace{2 em}
    \begin{tikzpicture}[scale = 0.75]
    \baseConfigurationdiamond{}{}{}{}
    
    \draw[gedge] (x)--++(120:0.5) (x)--++(60:0.5) (y)--++(240:0.5) (y)--++(300:0.5);
    \draw[gedge] (u)--++(150:0.5);
    \draw[gedge] (u)--++(210:0.5);
    
    \path[->]
    \chargepath{(u)}{dashedge,bend left}{above}{$1$}{(v)};
    
    \node at (0,-2.5) {(R6a)};
    \end{tikzpicture} \hspace{2 em}
    \begin{tikzpicture}[scale = 0.75]
    \confgurationC
    
    \draw[gedge] (u)--++(150:0.5);
    \draw[gedge] (u)--++(210:0.5);
    
    \path[->]
    \chargepath{(u)}{dashedge,bend left}{above}{$1$}{(v)};
    
    \node at (0,-2.5) {(R6a)};
    \end{tikzpicture} \hspace{2 em}
    \begin{tikzpicture}[scale = 0.75]
    \baseConfigurationdiamond{}{}{}{}
    
    \draw[gedge] (x)--++(120:0.5) (x)--++(60:0.5) (y)--++(240:0.5) (y)--++(300:0.5) (y)--++(260:0.5) (y)--++(280:0.5);
    \draw[gedge] (u)--++(150:0.5);
    \draw[gedge] (u)--++(210:0.5);

    \path[->]
    \chargepath{(u)}{dashedge,bend left}{above}{$1/2$}{(v)};   
    
    \node at (0.5,1.5) {$+$};
    \node at (0.5,-1.5) {$+$};
    
    \node at (0,-2.5) {(R6b)};
    \end{tikzpicture} \hspace{2 em}
        \begin{tikzpicture}[scale = 0.75]
        \confgurationE
        \path[->]
        \chargepath{(u)}{dashedge, bend right}{right}{$1/4$}{(x)}
        \chargepath{(u)}{dashedge, bend left}{right}{$1/4$}{(y)};
        
        \node at (1.5,0.5) {$+$};
        \node at (0,-2.5) {(R6c)};
    \end{tikzpicture} \hspace{2 em}
        \begin{tikzpicture}[scale = 0.75]
        \confgurationF
        \path[->]
        \chargepath{(u)}{dashedge, bend right}{right}{$1/2$}{(x)};   
    \node at (1.5,0.5) {$+$};
    \node at (0.5,-1.5) {$+$};
    \node at (0,-2.5) {(R6d)};
    \end{tikzpicture} \hspace{2 em}
    \begin{tikzpicture}[scale = 0.75]
    \confgurationG
    
    \draw[gedge] (u) -- ++(180:0.5);
    \draw[gedge] (v)--++(0:0.5);
    \draw[gedge] (x) -- ++(90:0.5);
    
    \path[->] 
    \chargepath{(y)}{dashedge, bend right}{near start, left}{$1/2$}{(x)};
    
    \node at (0,-2.5) {(R7a)};
    \end{tikzpicture} \hspace{2 em}
    \begin{tikzpicture}[scale = 0.75]
    \confgurationG
    
    \draw[gedge] (u) -- ++(180:0.5);
    \draw[gedge] (x)--++(60:0.5);
    \draw[gedge] (x)--++(120:0.5);
    \draw[gedge] (x) -- ++(90:0.5);
    
    \path[->]
    \chargepath{(y)}{dashedge, bend left}{left}{$1/2$}{(v)};
    \node at (0.5,1.5) {$+$};
    \node at (0,-2.5) {(R7b)};
    \end{tikzpicture} \hspace{2 em}
    \begin{tikzpicture}[scale = 0.75]
    \confgurationG
    
    \draw[gedge] (x)--++(60:0.5);
    \draw[gedge] (x)--++(120:0.5);  
    
    \draw[gedge] (u)--++(150:0.5);
    \draw[gedge] (u)--++(210:0.5);
    
    \path[->]
    \chargepath{(y)}{dashedge, bend left}{left}{$1/2$}{(v)};

    \node at (0,-2.5) {(R7c)};
    \end{tikzpicture} \hspace{2 em}
    \begin{tikzpicture}[scale = 0.75]
    \baseConfigurationdiamond{}{}{}{}
    \draw[gedge] (x)--++(60:0.5);
    \draw[gedge] (x)--++(120:0.5);  
    
    \draw[gedge] (u)--++(150:0.5);
    \draw[gedge] (u)--++(210:0.5);
    
    \draw[gedge] (y)--++(240:0.5) (y)--++(260:0.5) (y)--++(280:0.5) (y)--++(300:0.5);
    
    \path[->]
    \chargepath{(y)}{dashedge, bend left}{left}{$1/2$}{(v)};
    
    \node at (0.5,-1.5) {$+$};
    \node at (0,-2.5) {(R8a)};
    \end{tikzpicture} \hspace{2 em}
    \begin{tikzpicture}[scale = 0.75]
    \baseConfigurationdiamond{}{}{}{}
    \draw[gedge] (x)--++(90:0.5);
    
    \draw[gedge] (u)--++(180:0.5);
    \draw[gedge] (v)--++(0:0.5);
    
    \draw[gedge] (y)--++(240:0.5) (y)--++(260:0.5) (y)--++(280:0.5) (y)--++(300:0.5);
    
    \path[->]
    \chargepath{(y)}{dashedge, bend right}{near start, left}{$1/2$}{(x)};
    
    \node at (0.5,-1.5) {$+$};
    \node at (0,-2.5) {(R8b)};
    \end{tikzpicture} \hspace{2 em}
        \begin{tikzpicture}[scale = 0.75]
    \baseConfigurationdiamond{}{}{}{}
    \draw[gedge] (x)--++(60:0.5);
    \draw[gedge] (x)--++(120:0.5);
    
    \draw[gedge] (u)--++(180:0.5);

    \draw[gedge] (y)--++(240:0.5) (y)--++(260:0.5) (y)--++(280:0.5) (y)--++(300:0.5);
    
    \path[->]
    \chargepath{(y)}{dashedge, bend left}{left}{$1/2$}{(v)};
    
    \node at (0.5,1.5) {$+$};
    \node at (0.5,-1.5) {$+$};
    \node at (0,-2.5) {(R8c)};
    \end{tikzpicture} \hspace{2 em}
    \begin{tikzpicture}[scale = 0.75]
    \baseConfigurationdiamond{}{}{}{}
    
    \draw[gedge] (u)--++(180:0.5) (u)--++(150:0.5) (u)--++(210:0.5) (v) --++(0:0.5);
    \draw[gedge] (x)--++(90:0.5);
    \draw[gedge] (y)--++(270:0.5);
    \path[->]
    \chargepath{(u)}{dashedge, bend left}{right}{$1/2$}{(y)}
    \chargepath{(u)}{dashedge, bend right}{right}{$1/2$}{(x)};
    
    \node at (-1.5,0.5) {$+$};
    \node at (1.5,0.5) {$+$};
    \node at (0,-2.5) {(R9a)};
    \end{tikzpicture} \hspace{ 2 em}
    \begin{tikzpicture}[scale = 0.75]
    \baseConfigurationdiamond{}{}{}{}
    
    \draw[gedge] (u)--++(180:0.5) (u)--++(150:0.5) (u)--++(210:0.5);
    \draw[gedge] (v)--++(0:0.5);
    \draw[gedge] (x)--++(120:0.5) (x)--++(60:0.5);
    \draw[gedge] (y)--++(270:0.5);
    \path[->]
    \chargepath{(u)}{dashedge, bend left}{right}{$1$}{(y)};
  
    \node at (-1.5,0.5) {$+$};
    \node at (1.5,0.5) {$+$};
    \node at (0.5,1.5) {$+$};  
    \node at (0,-2.5) {(R9b)};
    \end{tikzpicture} \hspace{2 em}
    \begin{tikzpicture}[scale = 0.75]
    
    \baseConfigurationdiamond{}{}{}{}
    
    \draw[gedge] (u)--++(180:0.5) (u)--++(150:0.5) (u)--++(210:0.5);
    \draw[gedge] (x)--++(120:0.5) (x)--++(60:0.5);
    \draw[gedge] (y)--++(240:0.5) (y)--++(300:0.5);
    
    \path[->]
    \chargepath{(u)}{dashedge, bend left}{above}{$1$}{(v)};
    
    \node at (-1.5,0.5) {$+$};
    \node at (0.5,1.5) {$+$};
    \node at (0.5,-1.5) {$+$};
    \node at (0,-2.5) {(R9c)};
    \end{tikzpicture}
    \caption{Discharging rules for Theorem~\ref{main567}}
    \label{fig:my_label}
\end{figure}

\begin{claim}\label{facecheck}
The final charge of every face of $G$ is nonnegative. 
\end{claim}
\begin{proof}
Given that $G$ does not contain any faces of length $5,6$ or $7$, we consider $8^+$-faces, $4$-faces, and $3$-faces as three separate cases covering everything.

Suppose that $f$ is an $8^+$-face. 
Then the initial charge of $f$ is equal to $\ell(f) - 4$.
By (R1) and (R2), $f$ sends at most $\frac{1}{2}$ for each of these edge that is not a bridge and charge $1$ to each bridge by (R2). 
This means that $f$ sends at most $\left\lceil \frac{\ell(f)}{2} \right\rceil \leq \ell(f) - 4$ total charge. 
Since (R1) and (R2) are the only rules requiring an $8^+$-face to send out charge, every $8^+$-face has nonnegative final charge.

Suppose that $f$ is a $4$-face. Then $f$ has its initial charge $0$. Since $C_5$, $C_6$, and $C_7$ are forbidden subgraphs, $f$ must be incident with four $8^+$-faces. 
By (R1), each face sharing an edge with $f$ sends charge $\frac{1}{2}$ to $f$ for every edge they have in common, leaving $f$ with a total charge of $2$ before applying (R2)--(R9). 
Given that (C2) is a reducible configuration, $f$ cannot contain more than two $3$-vertices.
Thus, (R4) applies to $f$ at most twice, which decreases the charge at $f$ by at most 2.
Since no other rules apply to $4$-faces, $f$ has nonnegative final charge. 

Next suppose that $f$ is a $3$-face that is not contained in a diamond.
Every face incident to $f$ must be an $8^+$-face since $C_5$, $C_6$, and $C_7$ are forbidden subgraphs. 
This means that after applying (R1), $f$ has charge $\frac{1}{2}$. 
Among rules (R2)--(R9), only (R5) requires a $3$-face to send out charge. 
If (R5) applies to $f$, then it only requires $f$ to send a charge of $\frac{1}{2}$. 
This means that $f$ has nonnegative final charge. 

Lastly, assume that $f$ is a $3$-face contained in a diamond $D$.
Then $f$ shares one edge with another $3$-face. 
Since $C_5$, $C_6$, and $C_7$ are forbidden subgraphs, $f$ shares its other two edges with $8^+$-faces. 
By (R1), $f$ receives charge at least $\frac{1}{2}$ for each edge it shares with an $8^+$-face, leaving $f$ with charge at least $0$ before applying (R2)--(R9). 
None of these rules, however, demand charge from a $3$-face that is contained in a diamond, implying that $f$ will end with nonnegative charge. 
As we have considered all possible faces in $G$, this completes the proof of Claim~\ref{facecheck}.
\end{proof}

\begin{claim}\label{edgecheck}
The final charge of every edge of $G$ is nonnegative. 
\end{claim}
\begin{proof}
Let $e = uz$ be an edge of $G$. If $e$ is incident with a $3$-face or a $4$-face, then none of the rules apply to $e$ and there is nothing to prove. Otherwise, $e$ has charge $1$ after applying (R2). As (R3) is the only rule that requires any edge to send out charge, it suffices to verify that $e$ will never be asked to give more than $1$ charge under (R3). 

If $u$ and $z$ are both $3$-vertices, then only (R3a) applies to $e$ and the edge sends exactly $\frac{1}{2}$ to each of $u$ and $z$. If $u$ is a $3$-vertex and $z$ is a $4^+$-vertex, then only (R3b) applies, and $e$ send exactly $1$ to $u$. 

If $u$ is a $4^+$-vertex and $z$ is a $5^+$-vertex, then (R3) does not apply with $z$ and $e$ sends charge at most $1$ using either (R3c) or (R3d).

The remaining case is that both $u$ and $z$ are 4-vertices.
The rules demand $e$ to send charge more than 1 if by symmetry (R3c) applies with $u$ and one of (R3c) and (R3d) applies with $z$. However, this would give reducible configurations 
(C9) and (C12), respectively. 
Therefore, no edge in $G$ that begins with charge $1$ will ever be asked to send out more than $1$ total charge, completing the proof of Claim~\ref{edgecheck}. 
\end{proof}

\begin{claim}\label{4pluscheck}
The final charge of every $4^+$-vertex is nonnegative. 
\end{claim}
\begin{proof}
Suppose that $v$ is a $4$-vertex. The initial charge of $v$ is $0$, and there are no rules requiring $v$ to send out charge, so $v$ will end with nonnegative charge. 

Next suppose that $v$ is a $5$-vertex. Then the initial charge of $v$ is $1$. Only (R6) and (R7) require a $5$-vertex to distribute charge. Therefore, we may assume that $v$ is incident with at least one diamond. Given that $G$ does not contain any $C_5$, $C_6$, $C_7$, or $B_5$ subgraphs, $v$ is incident with at most two diamonds. 

First suppose that $v$ is incident with exactly one diamond $D$. If $v$ is a middle vertex of $D$ then only (R6) applies to $v$, and if $v$ is a side vertex of $D$ then only (R7) applies to $v$. 
As neither of these two rules will require $v$ to send out charge more than $1$, $v$ will end with nonnegative charge. 

Next suppose that $v$ is incident with two diamonds $D_1$ and $D_2$. 
Since that $G$ does not contain any $C_5$, $C_6$, $C_7$, or $B_5$ subgraphs, $D_1$ and $D_2$ must be edge disjoint. Since $d(v) = 5$, $v$ cannot be a middle vertex of both diamonds. If $v$ is a side vertex of both diamonds, then only (R7) applies to $v$. As (R7) will not require $v$ to send charge more than $\frac{1}{2}$ to either diamond, $v$ will end with nonnegative charge.

Therefore, we may assume that $v$ is a middle vertex of $D_1$ and a side vertex of $D_2$.
In this case, it is possible that both (R6) and (R7) apply to $v$. 
Among the subcases of (R6), only (R6a) requires $v$ to send out charge for more than $\frac{1}{2}$, and (R7) will never ask $v$ send out charge more than $\frac{1}{2}$. 
Given that configuration (C8) is reducible, (R6a) cannot apply with (R7a). 
Next, given that configuration (C10) is reducible, (R6a) cannot apply with (R7b). 
By assumption of (R7c), (R6a) cannot apply with (R7c). 
Therefore $v$ is never asked to send more than $1$, implying that $v$ will end with nonnegative charge. 

Now suppose that $v$ is a $6^+$-vertex. 
The only rules that apply to $v$ are (R8) and (R9). 
Under these rules, $v$ sends at most $1$ to all diamonds that contain $v$ as a middle vertex, and $v$ sends at most $1/2$ to all diamonds that contain $v$ as a side vertex. 
Assume that $v$ is the middle vertex of $k$ distinct diamonds, and incident to $m$ other faces of size $3$. By Lemma~\ref{deglem}, the final charge of $v$ is at least
$$d(v) - 4 - k - \frac{m}{2} = 
d(v)-4 - \frac{3k+2m}{4} -\frac{k}{4} \ge \frac{3d(v)}{4}-4 -  \frac{1}{4}\Big\lfloor\frac{d(v)}{3} \Big\rfloor \ge \frac{2d(v)}{3}-4,$$
and $\frac{2d(v)}{3}-4$ is nonnegative whenever $d(v) \geq 6$. This completes the proof of Claim~\ref{4pluscheck}. 
\end{proof}

\begin{claim}\label{3vert}
The final charge of every $3$-vertex that is not contained in a diamond is nonnegative. 
\end{claim}
\begin{proof}
Let $v$ be a $3$-vertex that is not contained in a diamond. 
Then the initial charge of $v$ is $-1$. 
As there are no rules requiring $v$ to send out charge, we only need to verify that $v$ will receive charge at least $1$. 
First suppose that $v$ is not incident to any $3$-faces or $4$-faces. Then each of the three edges incident to $v$ receive charge $1$ under (R2). 
Next, each of these edges sends $\frac{1}{2}$ to $v$ by (R3), leaving $v$ with a charge of $\frac{1}{2}$. 

Now suppose that $v$ is incident to at least one $4$-face $f$.
By (R4), $v$ receives $1$ from $f$ and we are done. 
Therefore, we may assume that $v$ is not incident to any $4$-face, and that $v$ is incident to at least one $3$-face $T$. 
By assumption, $T$  is not contained in a diamond. 
\begin{itemize}
    \item Case 1: $T$ contains another $3$-vertex. In this case, $v$ must be adjacent to a $4^+$-vertex $u$ that is not contained $T$, since (C2) is reducible. Then by (R3b), $v$ receives a charge of $1$ from the edge $uv$. 
    \item Case 2: $T$ contains two  $4^+$-vertices. Again, let $u$ be the neighbor of $v$ that in not contained in $T$. Here, $v$ will receive at least $\frac{1}{2}$ from (R3). With that being said, $T$ will send the remaining $\frac{1}{2}$ to $v$ under (R5). 
\end{itemize}
This completes the proof of Claim~\ref{3vert}. \end{proof}

\begin{claim}\label{3vertdiamond}
The final charge of every $3$-vertex that is incident to a diamond is nonnegative. 
\end{claim}
\begin{proof}
Assume that $v$ is a $3$-vertex incident to a diamond $D$.
Since (C2) and (C3) are reducible, there is at most one other $3$-vertex incident to $D$. 
Since the initial charge of $v$ is $-1$, and there is no rule requiring a $3$-vertex to send charge, it suffices to show that $v$ will always receive charge at least $1$ after applying rules (R1)--(R9). 
We consider the following cases. \medskip

\textbf{Case 1:} $v$ is the only $3$-vertex incident to $D$ and $v$ is a side vertex of $D$. 
Given the list of forbidden subgraphs in $G$, the other two faces incident to $v$ must be $8^+$-faces.
Hence by (R3), $v$ receives charge at least $\frac{1}{2}$ from the only edge incident to $v$ that is not a part of $D$. 
There are three subcases to Case 1 showing how $v$ gets another $\frac{1}{2}$ of charge.
\begin{enumerate}
    \item $D = Dia(4-4,4,3)$ 
    \\  Let $x$ and $y$ denote the two middle vertices of $D$. Since (C11) is reducible, each of the neighbors of $x$ and $y$ that are not contained in $D$ must be $4^+$-vertices. Therefore by (R3d), $v$ receives $\frac{1}{2}$ from the each of two edges incident to $x$ and $y$ that are not contained in $D$. %
    \item $D = Dia(4-4,5^+,3)$
    \\  %
    Let $u$ denote the other side vertex in $D$. If $u$ is a $5$-vertex, then $v$ receives $\frac{1}{2}$ from $u$ by (R7a). If $u$ is a $6^+$-vertex, then $v$ receives $\frac{1}{2}$ from $u$ by (R8b). 
    \item $D = Dia(5^+-4^+,4^+,3)$
    \\  %
    Let $u$ denote the $5^+$-vertex that is the middle vertex of $D$. If $d(u) = 5$, then $u$ sends charge $\frac{1}{2}$ to $v$ by (R6d). If $d(u) \geq 6$, then $u$ sends charge $\frac{1}{2}$ by (R9b). 
\end{enumerate}
In all three cases, the final charge of $v$ is nonnegative.
\medskip

\textbf{Case 2:} $v$ is the only $3$-vertex incident to $D$ and $v$ is a middle vertex of $D$. 
There are four subcases to Case 2. In each $v$ receives charge $1$ which leads to nonnegative final charge.
\begin{enumerate}
    \item $D = Dia(4-3,4,4^+)$
    \\ Let $u$ denote the middle $4$-vertex of $D$. Since (C4) is reducible, the unique neighbor $z$ of $u$ not contained in $D$ must be a $4^+$-vertex. Therefore, the edge $uz$ will send charge $1$ to $v$ by (R3c), leaving $v$ with nonnegative charge. 
    
    \item $D = Dia(4-3,5^+,5^+)$
    \\ Let $x$ and $y$ denote the two side vertices of $D$. If $d(x) = 5$, then $x$ will send charge $\frac{1}{2}$ to $v$ by (R7b). If $d(x) \geq 6$, then $x$ will send charge $\frac{1}{2}$ to $v$ by (R8c). As the rules apply to $y$ identically as they do to $x$, it follows that $v$ will end with nonnegative charge. 
    
    \item $D = Dia(5-3,4^+,4^+)$. 
    \\ Let $u$ denote the middle $5$-vertex of $D$ and let $x$ and $y$ denote each of the side $4^+$-vertices of $D$ with $d(x) \leq d(y)$.
    If $d(x) = d(y) = 4$, then $u$ sends charge $1$ to $v$ by (R6a). If $d(x) = 4$ and $d(y) = 5$, then $u$ sends charge $1$ to $v$ by (R6a).
    If $d(x) = 4$ and $d(y) \geq 6$, then $u$ sends charge $\frac{1}{2}$ to $v$ by (R6b) and $y$ sends $\frac{1}{2}$ to $v$ by  (R8a). 
    
    If $d(x) = d(y) = 5$, then $u$ sends charge $\frac{1}{2}$ to $v$ by (R6b). Since (C13) is reducible, $x$ and $y$ will not both send charge $1$ to a vertex of a different diamond under rule (R6a). Therefore, $v$ receives charge $\frac{1}{2}$ from either $x$ or $y$ by (R7c). If $d(y) \geq 6$, $y$ sends charge $\frac{1}{2}$ to $v$ by (R8a) and $u$ sends charge $\frac{1}{2}$ to $v$ by (R6b). This leaves $v$ with nonnegative charge.
    
    \item $D = Dia(6^+-3,4^+,4^+)$
    \\ Let $u$ denote the middle vertex of $D$. Then $u$ sends charge $1$ to $v$ by (R9c). This leaves $v$ with nonnegative charge. 
\end{enumerate}

\textbf{Case 3:} There are two $3$-vertices incident to $D$, one of which is $v$. Let $x$ denote the other $3$-vertex incident to $D$. Since (C5) and (C6) are reducible configurations, both $x$ and $v$ are side verties of $D$. Since (C7) is reducible, we may assume that if one of the middle vertices of $D$ is a $4$-vertex, then the other middle vertex is a $6^+$-vertex. There are two subcases to Case 3. 
\begin{enumerate}
    \item $D = Dia(6^+-4,3,3)$. 
    \\ Let $u$ denote the $6^+$-vertex incident to $D$. By (R9a), $u$ sends charge $\frac{1}{2}$ to $v$. Since $v$ is a side vertex of $D$, and $d(v) = 3$, it follows that $v$ is incident to exactly one edge that is not contained in $D$. By (R3), $v$ will receive charge at least $\frac{1}{2}$ from this edge, leaving $v$ with nonnegative charge.
    The case of $x$ is symmetric.
    \item $D = Dia(5^+-5^+,3,3)$
    \\ Since $v$ is a side vertex of $D$, and $d(v) = 3$, it follows that $v$ is incident to exactly one edge that is not contained in $D$. By (R3), $v$ will receive charge at least $\frac{1}{2}$ from this edge.
    
    Let $a$ and $b$ denote the middle vertices of $D$. If $d(a) = 5$, then $v$ receives charge $\frac{1}{4}$ from $a$ by (R6c). If $d(a) \geq 6$, then $v$ receives charge $\frac{1}{2}$ from $a$ by (R9a). As the rules apply to $b$ identically as they do to $a$, it follows that $v$ will end with nonnegative charge. 
    Again, the case of $x$ is symmetric.
\end{enumerate}
Since we have covered all cases where $v$ is contained in a diamond, this completes the proof of Claim~\ref{3vertdiamond}. 
\end{proof}
\medskip

Claims~\ref{facecheck}--\ref{3vertdiamond} show that the final charge of every vertex, face, and edge is nonnegative. Hence the sum of the charges is also nonnegative, which is a contradiction with the sum of the initial charges being $-8.$
This finishes the proof of Lemma~\ref{lem:unavoidable B5}.
\end{proof}

\section{Proof of Theorem \ref{main567weak}}

\subsection{Reducible configurations}

We will use the following list of enhanced weakly reducible configurations. See Figure~\ref{fig:reducibleA} for illustration of these configurations. 
\begin{itemize}
\item[(D1)] A vertex of degree at most 2. 
\item[(D2)] $T(3,3,3)$ and $T(3,3,4)$. %
\item[(D3)] Two diamonds $D_1 = Dia(6-3,4,3)$ and $D_2 = Dia(6-3,4,4)$ sharing a middle $6$-vertex.
\item[(D4)]
$Dia(3-3,4^+,4^+)$ where the side vertices are in the boundary. 
\item[(D5)]
$Dia(4-5,3,3)$.

\item[(D6)] 
$Dia(4-3,4,4)$.
\item[(D7)]
$Dia(5-3,4,4)$ with another 3-vertex adjacent to the 5-vertex.

\item[(D8)]
$Dia(5-5,3,3)$ with another 3-vertex adjacent to one of the 5-vertices.

\item[(D9)] Three $3$-vertices $u,v,w$ such that $uv$ and $vw$ are edges, and $u$ and $w$ are independent.

\item[(D10)]
$Dia(5-3,4,3)$.

\item[(D11)]
$T(5,3,3)$ with another $3$-vertex adjacent to the $5$-vertex.

\item[(D12)]
Two triangles $T_1 = T_2 = T(6,3,3)$ sharing the $6$-vertex. 
\end{itemize}

In the next section we will prove the following theorem, showing  that (D1)--(D12) are unavoidable.
We remark that no identification of vertices in (D1)--(D12) is possible since otherwise, it creates a forbidden subgraph.
It is possible that some external edges can be identified in (D8), (D9), and (D11). We explicitly list those cases in Figure~\ref{fig:reducibleA} as (D8'), (D9'), (D11'), and (D11'').

\begin{thm}
\label{thm:dischargelagebook}
Every $\{K_4, C_5, C_6, C_7\}$-free planar graph  contains one of {\em (D1)--(D12)}.
\end{thm}

Let $\mathcal{F}=\{K_4, C_5, C_6, C_7,B_\ell\}$ for any fixed $\ell$. We will show that (D2), (D3), and (D5)--(D12) are enhanced weakly $(\mathcal{F},4)$-boundary-reducible configurations.
We will also show that (D4) is only a weakly $(\mathcal{F},4)$-boundary-reducible configuration.
However, the only neighbors of the vertices in the reducible part of (D4) are the non-adjacent vertices in the boundary, and all non-adjacent pairs that do not forbid $\mathcal F$ in each of (D1)--(D12) form a loose sets, implying that the condition required by the enhanced weak resolution is satisfied.

In order to use Lemma~\ref{lem:weak}, we want to have an enhanced weak $(\mathcal{F},4,b,\beta)$-resolution for some $\beta$ and $b$.
First, we check the condition (TIGHT) in the following lemma.
\begin{lemma}
Let $G$ be an $\mathcal{F}$-free graph containing $H$, where $H$ is one of (D2)--(D12). 
The number of $H$-tight vertices is at most $\beta \leq 10\ell$.
\end{lemma}
\begin{proof}
Let $H$ be one of (D2)--(D12) and  $v$ be an $H$-tight vertex adjacent to $u$ and $w$ in $H$.
First, suppose that $u$ and $w$ are not adjacent.
There are no non-edges in (D2) and (D4).
The non-edge in (D9) forms a loose set.
By inspection of each pair of non-adjacent verticese in (D5)--(D8) and  (D10)--(D12), we observed that if $v$ was adjacent to any of these pairs, we would obtain a $C_5$ or a $C_6$, contradicting that $G$ is $\mathcal{F}$-free.

Second, suppose that $uw$ is an edge.
As $B_\ell$ is in $\mathcal{F}$, the number of $H$-tight vertices for $uv$ is at most $\ell-3$.
Since $H$ is one of (D2)--(D12), it has at most $10$ edges. 
This bounds that the total number of $H$-tight vertices as $\beta\leq 10(\ell-3) \leq  10\ell$.
\end{proof}

\begin{lemma}
  The configurations (D2), (D3), (D5)--(D12) are enhanced weakly $(\mathcal{F},4)$-boundary-reducible. See Figure~\ref{fig:reducibleA} for illustration.
\end{lemma}
\begin{proof}
Given the rules for enhanced weak $(\mathcal{F},k)$-boundary reducibility, it is straightforward to verify that configurations (D2)--(D12) are reducible. 
We also provide a computer program at \oururl{}\footnote{This program is also available as a part of the sources in our arXiv submission.} to do so. One notable difference is that the greedy algorithm is not always sufficient. We also added test for Gallai tree, which helped. Just greedy algorithm could end with a diamond, where middle vertices have lists $L$ of size 3 and side vertices would have lists of size 2, which is not a Gallai tree and hence it is $L$-colorable, but not in a greedy way.

In order to highlight the difference between regular and weak reducibility, we will give a short proof that (D10) is weakly $(\mathcal{F},k)$-boundary reducible, but not $(\mathcal{F},k)$-boundary reducible. Let $R$ be a subgraph of a graph $G$ defined by configuration (D10). Let $a,b,c$ and $d$ be vertices of $R$. The initial list sizes of a list assignment $L$ as defined by the function $(4 - (\text{deg}_G + \deg_R))$ are given in Figure~\ref{reduce:example:Dten}. 
\begin{figure}[H]
    \centering
        \begin{tikzpicture}[scale = 0.75]
    \baseConfigurationdiamond{}{}{}{}
    
    \draw[gedge] (x)--++(120:0.5);
    \draw[gedge] (x)--++(60:0.5);
    \draw[gedge] (v)--++(30:0.5);
    \draw[gedge] (v)--++(-30:0.5);
    \draw[gedge] (y)--++(270:0.5);
    
    \node at (-0.35,-1.5) {$d$};
    \node at (-0.35,1.5) {$b$};
    \node at (-1.5,0.35) {$a$};
    \node at (1.5,0.35) {$c$};
    
    \node at (0.35,-1.5) {$3$};
    \node at (0.35,1.5) {$2$};
    \node at (-1.5,-0.35) {$4$};
    \node at (1.5,-0.35) {$2$};
    \end{tikzpicture}
    \caption{Configuration (D10)}
    \label{reduce:example:Dten}
\end{figure}
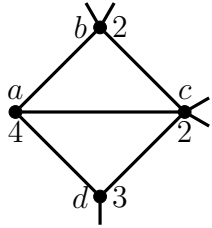
First we will show that we cannot fix the color of $a$ and still properly color $R$. Indeed, if the lists of $b$ and $c$ are identical and both contained the color assigned to $a$, there would be no proper $L$-coloring of $R$. 

That being said, if we fix the color of any other vertex in $R$, then we will still be able to properly $L$-color $R$. Therefore, we can only apply (FIX) to a subset of the vertices of $R$. Given the graphs in $\mathcal{F}$, it immediately follows that the graph $H = R$ is weakly $(\mathcal{F},k)$-boundary reducible, but as we have show, it is not $(\mathcal{F},k)$-boundary reducible.

In the enhanced version, the main trick is that we never need to check (FORB) on two adjacent vertices. We do that by allowing (FIX) only on vertices, where their external neighbors are non-adjacent. The easiest way to do so is to use (FIX) only vertices that have at most 1 outside neighbor. 
In case of (D10), the only option for (FIX) is the vertex $d$.
\end{proof}

The above works in all cases except (D1) and (D4). 
As the case (D1) was already discussed, we now justify the usage of the configuration (D4).

\begin{lemma}
  The configuration (D4) is weakly $(\mathcal{F},4)$-boundary-reducible and for all vertices $x$ in its reducible part holds that $|N(x)\cap R_j|\le 1$ or $N(x)\cap R_j$ is a loose set, where $R_j$ is a reducible part of some other configuration.
\end{lemma}
\begin{proof}
The check of the reducibility is straightforward.

The only vertices adjacent to the vertices in the reducible part are the two vertices in the boundary which are non-adjacent as $K_4\in\mathcal F$.
Therefore, it is sufficient to check non-adjacent non-$\mathcal F$-forbidding vertices in configurations (D1)--(D12).
As was already discussed the only such a non-edge is the non-edge in (D9) which forms a loose set.
\end{proof}

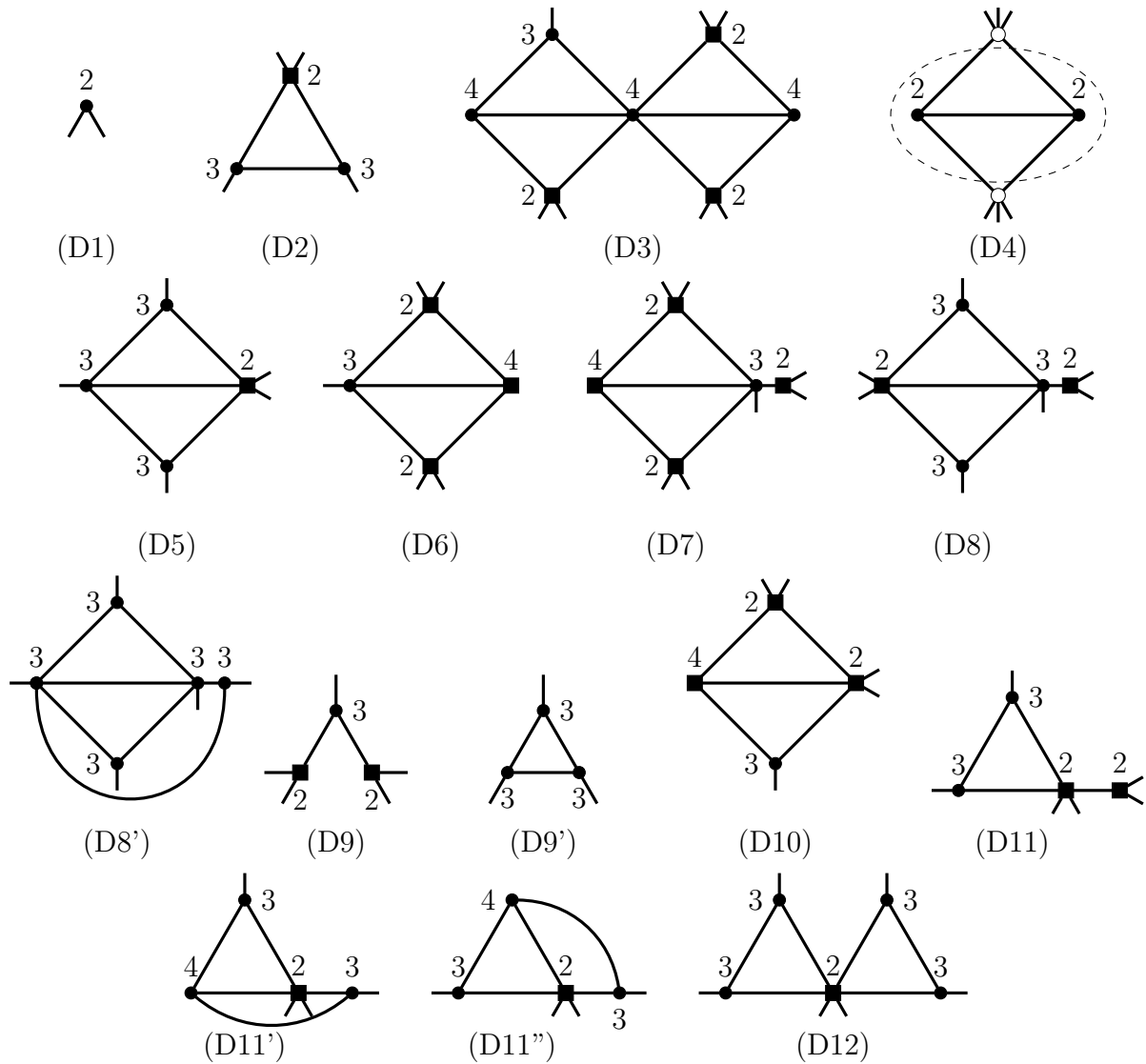
\begin{figure}[H]
    \centering
\begin{tikzpicture}
\node[vtx,label = $2$] (v) at (0,1) {};
\draw[gedge] (v) -- ++(240:0.5) {};
\draw[gedge] (v) -- ++(300:0.5) {};

\node at (0,-1) {(D1)};
\end{tikzpicture}\hspace{2 em}
\begin{tikzpicture}[scale = 0.75]
\baseConfigurationtriangle{$3$}{$2$}{$3$}

\draw[gedge] (w)--++(240:0.5) (u)--++(120:0.5) (u)--++(60:0.5) (v)--++(-60:0.5);
\path (w) ++(60:2) node (u)[vtxNoFIX] {};
\node at (2,-1.5) {(D2)};
\end{tikzpicture}\hspace{ 2 em }
    \begin{tikzpicture}[scale = 0.75]
    \baseConfigurationdiamond{$4$}{$4$}{$3$}{$2$}
    
    \node[vtxNoFIX,label = right:$2$] at (3,1.5) (a) {};
    \node[vtxNoFIX, label = right:$2$] at (3,-1.5) (b) {};
    \node[vtx,label = above:$4$] at (4.5,0) (c) {};
    
    \draw[gedge] (v)--(a) (v)--(b) (a)--(c) (b)--(c) (v)--(c);
    \draw[gedge] (x)--++(90:0.5) (y)--++(240:0.5) (y)--++(300:0.5) (a)--++(120:0.5) (a)--++(60:0.5) (b)--++(240:0.5) (b)--++(300:0.5);
 
    \node[vtxNoFIX] at (0,-1.5) {};
 
    \node at (1.5,-2.5) {(D{3})};
\end{tikzpicture} \hspace{2 em}
\begin{tikzpicture}[scale = 0.75]
\baseConfigurationdiamond{$2$}{$2$}{}{}
\draw[gedge] (x)--++(120:0.5);
\draw[gedge] (x)--++(60:0.5);
\draw[gedge] (x)--++(90:0.5);

\draw[gedge] (y)--++(270:0.5);
\draw[gedge] (y)--++(300:0.5);
\draw[gedge] (y)--++(240:0.5);

\draw[dashed] (0,0) ellipse[x radius = 2 cm, y radius = 1.25cm];
\node[bndry] at (x) {};
\node[bndry] at (y) {};
 
\node at (0,-2.5) {(D4)};
\end{tikzpicture}\hspace{2 em}
    \begin{tikzpicture}[scale = 0.75]
    \baseConfigurationdiamond{$3$}{$2$}{$3$}{$3$}
    
    \draw[gedge] (x)--++(90:0.5);
    \draw[gedge] (u)--++(180:0.5);
    \draw[gedge] (v)--++(30:0.5);
    \draw[gedge] (v)--++(-30:0.5);
    \draw[gedge] (y)--++(270:0.5);
    
    \node[vtxNoFIX] at (1.5,0) {};
    
    \node at (0,-3) {(D5)};
    \end{tikzpicture}
    \hspace{1 em}
    \begin{tikzpicture}[scale = 0.75]
    \baseConfigurationdiamond{$3$}{$4$}{$2$}{$2$}
    
    \draw[gedge] (x)--++(120:0.5);
    \draw[gedge] (x)--++(60:0.5);
    \draw[gedge] (u)--++(180:0.5);
    \draw[gedge] (y)--++(240:0.5);
    \draw[gedge] (y)--++(300:0.5);
    
    \node[vtxNoFIX] at (1.5,0) {};
    \node[vtxNoFIX] at (0,-1.5) {};
    \node[vtxNoFIX] at (0,1.5) {};
    \node at (0,-3) {(D6)};
    \end{tikzpicture}
    \hspace{1em}
    \begin{tikzpicture}[scale = 0.75]
    \baseConfigurationdiamond{$4$}{$3$}{$2$}{$2$}
    
    \node[vtxNoFIX,label = $2$] (z) at (2,0) {};
    \draw[gedge] (v)--(z);
    
    \draw[gedge] (x)--++(120:0.5);
    \draw[gedge] (x)--++(60:0.5);
    \draw[gedge] (y)--++(240:0.5);
    \draw[gedge] (y)--++(300:0.5);
    \draw[gedge] (v)--++(270:0.5);
    \draw[gedge] (z)--++(30:0.5);
    \draw[gedge] (z)--++(-30:0.5);
    
    \node[vtxNoFIX] at (-1.5,0) {};
    \node[vtxNoFIX] at (0,-1.5) {};
    \node[vtxNoFIX] at (0,1.5) {};
    
    \node at (0,-3) {(D7)};
    \end{tikzpicture}
    \hspace{1em}
    \begin{tikzpicture}[scale = 0.75]
    \baseConfigurationdiamond{$2$}{$3$}{$3$}{$3$}
    
    \node[vtxNoFIX,label = $2$] (z) at (2,0) {};
    \draw[gedge] (v)--(z);
    
    \draw[gedge] (u)--++(150:0.5);
    \draw[gedge] (u)--++(210:0.5);
    \draw[gedge] (y)--++(270:0.5);
    \draw[gedge] (v)--++(270:0.5);    
    \draw[gedge] (x)--++(90:0.5);
    \draw[gedge] (z)--++(30:0.5);
    \draw[gedge] (z)--++(-30:0.5);
    
    \node[vtxNoFIX] at (-1.5,0) {};
    
    \node at (0,-3) {(D8)};
    \end{tikzpicture}
    \hspace{1em}
    \begin{tikzpicture}[scale = 0.75]
    \baseConfigurationdiamond{$3$}{$3$}{$3$}{$3$}
    
    \node[vtx,label = $3$] (z) at (2,0) {};
    \draw[gedge] (v)--(z);
    
    \draw[gedge] (u)--++(180:0.5);
    \draw[gedge] (y)--++(270:0.5);
    \draw[gedge] (v)--++(270:0.5);    
    \draw[gedge] (x)--++(90:0.5);
    \draw[gedge] (z)--++(0:0.5);

    \draw[gedge] (z) to[out=270,in=270,looseness=2] (u);
    
    \node[vtx] at (-1.5,0) {};
    
    \node at (0,-3) {(D8')};
    \end{tikzpicture}    
    \begin{tikzpicture}
\node[vtxNoFIX,label = below:$2$] (w) at (1,0) {};
\path (w) ++(60:1) node (u)[vtx,label = right:$3$] {};
\path (w) ++(0:1) node (v)[vtxNoFIX,label=below:$2$] {};

\draw[gedge] (u) -- (v);
\draw[gedge] (w) -- (u);

\draw[gedge] (u) -- ++(90:0.5);
\draw[gedge] (v) -- ++(300:0.5);
\draw[gedge] (w) -- ++(240:0.5);
\draw[gedge] (w) -- ++(180:0.5);
\draw[gedge] (v) -- ++(0:0.5); 

\node at (1.5,-1) {(D9)};
\end{tikzpicture}
\hspace{2 em}
    \begin{tikzpicture}
\node[vtx,label = below:$3$] (w) at (1,0) {};
\path (w) ++(60:1) node (u)[vtx,label = right:$3$] {};
\path (w) ++(0:1) node (v)[vtx,label=below:$3$] {};

\draw[gedge] (u) -- (v);
\draw[gedge] (w) -- (u);

\draw[gedge] (u) -- ++(90:0.5);
\draw[gedge] (v) -- ++(300:0.5);
\draw[gedge] (w) -- ++(240:0.5);
\draw[gedge] (w) -- (v);

\node at (1.5,-1) {(D9')};
\end{tikzpicture} \hspace{2 em}
\begin{tikzpicture}[scale = 0.75]
    \baseConfigurationdiamond{$4$}{$2$}{$2$}{$3$}
    
    \draw[gedge] (x)--++(120:0.5);
    \draw[gedge] (x)--++(60:0.5);
    \draw[gedge] (v)--++(30:0.5);
    \draw[gedge] (v)--++(-30:0.5);
    \draw[gedge] (y)--++(270:0.5);
    
    \node[vtxNoFIX] at (-1.5,0) {};
     \node[vtxNoFIX] at (1.5,0) {};
      \node[vtxNoFIX] at (0,1.5) {};
    \node at (0,-3) {(D{10})};
    \end{tikzpicture}
    \hspace{1 em}
    \begin{tikzpicture}[scale = 0.75]
    \node[vtx,label = above:$3$] (w) at (1,0) {};
    \path (w) ++(60:2) node (u)[vtx,label = right:$3$] {};
    \path (w) ++(0:2) node (v)[vtxNoFIX,label = above:$2$] {};

    \path (v) ++(0:1) node (x)[vtxNoFIX,label = above:$2$] {}; 
    \draw[gedge] (u) -- (v);
    \draw[gedge] (w) -- (u);
    \draw[gedge] (w) -- (v);
    
    \draw[gedge] (v)--++(240:0.5);
    \draw[gedge] (v)--++(300:0.5);
    \draw[gedge] (u)--++(90:0.5);
    \draw[gedge] (w)--++(180:0.5);
    \draw[gedge] (x)--++(30:0.5);
    \draw[gedge] (x)--++(-30:0.5);
    \draw[gedge] (v)--(x);
    
    \node at (2,-1) {(D{11})};
    \end{tikzpicture}
    \hspace{1 em}
    \begin{tikzpicture}[scale = 0.75]
    \node[vtx,label = above:$4$] (w) at (1,0) {};
    \path (w) ++(60:2) node (u)[vtx,label = right:$3$] {};
    \path (w) ++(0:2) node (v)[vtxNoFIX,label = above:$2$] {};

    \path (v) ++(0:1) node (x)[vtx,label = above:$3$] {}; 
    \draw[gedge] (u) -- (v);
    \draw[gedge] (w) -- (u);
    \draw[gedge] (w) -- (v);
    
    \draw[gedge] (v)--++(240:0.5);
    \draw[gedge] (v)--++(300:0.5);
    \draw[gedge] (u)--++(90:0.5);
    \draw[gedge] (x)--++(0:0.5);
    \draw[gedge] (v)--(x);
    
    \draw[gedge] (x) to[bend left=40] (w);
    
    \node at (2,-1) {(D{11}')};
    \end{tikzpicture}
    \hspace{1 em}
    \begin{tikzpicture}[scale = 0.75]
    \node[vtx,label = above:$3$] (w) at (1,0) {};
    \path (w) ++(60:2) node (u)[vtx,label = left:$4$] {};
    \path (w) ++(0:2) node (v)[vtxNoFIX,label = above:$2$] {};

    \path (v) ++(0:1) node (x)[vtx,label = below:$3$] {}; 
    \draw[gedge] (u) -- (v);
    \draw[gedge] (w) -- (u);
    \draw[gedge] (w) -- (v);
    
    \draw[gedge] (v)--++(240:0.5);
    \draw[gedge] (v)--++(300:0.5);
    \draw[gedge] (w)--++(180:0.5);
    \draw[gedge] (x)--++(0:0.5);
    \draw[gedge] (v)--(x);
    
        \draw[gedge] (x) to[bend right=40] (u);
    
    \node at (2,-1) {(D{11}'')};
    \end{tikzpicture}
    \hspace{1 em}
    \begin{tikzpicture}[scale = 0.75]
    \node[vtx,label = above:$3$] (w) at (1,0) {};
    \path (w) ++(60:2) node (u)[vtx,label = left:$3$] {};
    \path (w) ++(0:2) node (v)[vtxNoFIX, label = above:$2$] {};

    \path (v) ++(0:2) node (x)[vtx,label = above:$3$] {}; 
    \path (v) ++(60:2) node (y)[vtx,label = right:$3$] {};
    \draw[gedge] (u) -- (v);
    \draw[gedge] (w) -- (u);
    \draw[gedge] (w) -- (v);
    
    \draw[gedge] (v)--++(240:0.5);
    \draw[gedge] (v)--++(300:0.5);
    \draw[gedge] (u)--++(90:0.5);
    \draw[gedge] (w)--++(180:0.5);
    \draw[gedge] (x)--++(0:0.5);
    \draw[gedge] (v)--(x);
    \draw[gedge] (v)--(y);
    \draw[gedge] (x)--(y);
    \draw[gedge] (y)--++(90:0.5);
    
    \node at (3,-1) {(D{12})};
    \end{tikzpicture}
\caption{Reducible configurations for Theorem~\ref{main567weak}. The labels give the list sizes remaining after accounting for the external neighbors and boundary vertices. The vertices whose colors cannot be fixed are drawn as squares. These vertices cannot be fixed because either their coloring does not extend or they have two external neighbors, with the one exception being (D4).   }\label{fig:reducibleA}
    
\end{figure}
\subsection{Discharging rules}

In this section, we prove the following Lemma~\ref{lem:unavoidable Bk}, that makes Theorem~\ref{main567weak} a corollary of Lemma~\ref{lem:weak}.

\begin{lemma}\label{lem:unavoidable Bk}
Let $G$ be a connected $\{K_4,C_5,C_6,C_7\}$-free plane graph.
Then $G$ contains at least one of the reducible configurations \emph{\textrm{(D1)--(D13)}}.
\end{lemma}
\begin{proof}
Assume for contradiction that $G$ is a $\{K_4,C_5,C_6,C_7\}$-free plane graph with no (D1)--(D12). We will use discharging to arrive to a contradiction. 

For every vertex $v$ assign the initial charge $ch(v) := 2\deg(v)-6$, and every face $f$ assign $ch(f) := \ell(f)-6$, where $\ell(F)$ is the length of the facial walk around $f$.
By Euler's formula, the total initial charges of all vertices and faces is $-12$. 
We sequentially apply the following rules that transfer charge. 
The charge after applying all the rules is called the \emph{final charge}.
We will show that the final charge is nonnegative for every vertex and every face, which is a contradiction with the total sum of all charges being $-12$.

\begin{itemize}
\item[(R1)] Every $8^+$-face sends charge $\frac{1}{4}$ to every incident $3$-face and $4$-face for every edge they have in common.

\item[(R2)] 
For every $3$-vertex $v$ that is incident to a triangle $t$ and an edge $uv$ that is not part of any triangle, the following applies. The two faces\footnote{may be the same face twice if $uv$ is a bridge} that are incident to $uv$,  each send the following charge to $t$:
\begin{enumerate}
    \item[(R2a)] $\frac{1}{8}$  if $\deg(u) = 3$,  
    \item[(R2b)] $\frac{1}{4}$ if $\deg(u) \geq 4$.
\end{enumerate} 

\item[(R3)] Every $4$-vertex sends charge 1 to every 3-face and 4-face adjacent to it. 
\item[(R4)] Every $5$-vertex sends charge 1 to every 4-face adjacent to it.
\item[(R5)] Every 5-vertex that is a middle vertex in
$Dia(5-3,4,4)$,
$Dia(5-3,3,5^+)$, or
$Dia(5-5,3,3)$
sends charge 1.5 to every 3-face of such diamond.

\item[(R6)] Every 5-vertex $v$, where rule (R5) does not apply, sends charge 1 to every 3-face of a diamond having $v$ as a middle vertex. 
\item[(R7)] 
For every 3-face $f=\{v,u,w\}$ and $5$-vertex $v$ such that $f$ is not part of a diamond having $v$ as a middle vertex, the following applies.
\begin{enumerate}
    \item[(R7a)] If both $\deg(u)\ge 4$ and $\deg(w)\ge 4$, then $v$ sends charge 1 to  $f$. 
    \item[(R7b)] Otherwise $v$ sends charge 2 to  $f$. 
\end{enumerate}

\item[(R8)] 
Every $6^+$-vertex $v$ sends charge 1 to every 4-face adjacent to it, and 2 to every 3-face $f$ adjacent to it, unless $f$ is part of a diamond having $v$ as a middle vertex.

\item[(R9)]

Every 6-vertex $v$ sends charge $1.75$ to every 3-faces of every $Dia(6-3,3,4)$ that contains $v$.

\item[(R10)] Every 6-vertex $v$ sends charge 1.5 to each 3-face of $Dia(6-3,4,4)$ that contains $v$.

\item[(R11)] Every 6-vertex $v$ sends charge 1.25 to every 3-face of any diamond $d$ having $v$ as a middle vertex,
where (R9) and (R10) did not apply.

\item[(R12)] Every $7^+$-vertex $v$ sends charge 1.75 to every 3-face of any diamond having $v$ as a middle vertex.

\item[(R13)] For every two 3-faces $f,g$ that form a diamond, if $g$ has positive charge while $f$ has negative charge, then $g$ gives $f$ all its positive charge. 
\end{itemize}

\begin{figure}[H]
    \centering
        \begin{tikzpicture}[scale = 0.75]
    \baseConfigurationtriangleUnlabeled
    \draw (v) node[vtx,label=right:$v$]{};
    \draw[gedge] (v)--++(240:0.5);
    \draw[gedge] (v)--++(300:0.5);
    
    \path[->]
    \chargepath{(v)}{dashedge}{}{}{(2,0.5)}; 
    
    \node at (1.8,0.6) {\footnotesize{$1$}};
    \node at (2,-1) {(R3)};
    \end{tikzpicture} \hspace{ 1 em}
    \begin{tikzpicture}[scale = 0.75]
    \baseConfigurationsquareUnlabeled
    \draw (v) node[vtx,label=left:$v$]{};
    \draw[gedge] (v)--++(240:0.5);
    \draw[gedge] (v)--++(300:0.5);
    
    \path[->]
    \chargepath{(v)}{dashedge}{}{}{(0,0)};
    
    \node at (0.2,0.2) {\footnotesize{$1$}};
    \node at (0,-2) {(R3)};
    \end{tikzpicture} \hspace{1 em}
    \begin{tikzpicture}[scale = 0.75]
    \baseConfigurationsquareUnlabeled
    \draw (v) node[vtx,label=left:$v$]{};
    \draw[gedge] (v)--++(240:0.5);
    \draw[gedge] (v)--++(300:0.5);
    \draw[gedge] (v)--++(270:0.5);
    
    \path[->]
    \chargepath{(v)}{dashedge}{}{}{(0,0)};
    
    \node at (0.2,0.2) {\footnotesize{$1$}};
    \node at (0,-2) {(R4)};
    \end{tikzpicture} \hspace{ 1 em }
    \begin{tikzpicture}[scale = 0.5]
    \node[vtx, label = $v$] (v) at (-2,0) {};
    \node[vtx] (u) at (2,0) {};
    \node[vtx] (x) at (0,2) {};
    \node[vtx] (y) at (0,-2) {};
    
    \draw [gedge] (u) -- (v)  (u) -- (x)  (u) -- (y)  (v) -- (x) (v) -- (y);
    
    \draw[gedge] (x) -- ++(60:0.75);
    \draw[gedge] (x) -- ++(120:0.75);
    
    \draw[gedge] (y) -- ++(240:0.75);
    \draw[gedge] (y) -- ++(300:0.75);
    
    \draw[gedge] (v) -- ++(150:0.75);
    \draw[gedge] (v) -- ++(210:0.75);
    \path[->]
    \chargepath{(v)}{dashedge}{}{}{(0,0.5)}
    \chargepath{(v)}{dashedge}{}{}{(0,-0.5)};
    
    \node at (0.5,0.5) {\footnotesize{$1.5$}};
    \node at (0.5,-0.5) {\footnotesize{$1.5$}};
    \node at (0,-4) {(R5)};
    \end{tikzpicture} \hspace{1 em}
    \begin{tikzpicture}[scale = 0.5]
    \node[vtx, label = $v$] (v) at (-2,0) {};
    \node[vtx] (u) at (2,0) {};
    \node[vtx] (x) at (0,2) {};
    \node[vtx] (y) at (0,-2) {};
    
    \draw [gedge] (u) -- (v)  (u) -- (x)  (u) -- (y)  (v) -- (x) (v) -- (y);
    
    \draw[gedge] (x) -- ++(90:0.75);
    
    \draw[gedge] (y) -- ++(240:0.75);
    \draw[gedge] (y) -- ++(270:0.75);
    \draw[gedge] (y) -- ++(300:0.75);
    
    \node at (-0.5,-2) {$+$};
    
    \draw[gedge] (v) -- ++(150:0.75);
    \draw[gedge] (v) -- ++(210:0.75);
    \path[->]
    \chargepath{(v)}{dashedge}{}{}{(0,0.5)}
    \chargepath{(v)}{dashedge}{}{}{(0,-0.5)};

    \node at (0.5,0.5) {\footnotesize{$1.5$}};
    \node at (0.5,-0.5) {\footnotesize{$1.5$}};
    
    \node at (0,-4) {(R5)};
    \end{tikzpicture} \hspace{1 em}
    \begin{tikzpicture}[scale = 0.5]
    \node[vtx, label = $v$] (v) at (-2,0) {};
    \node[vtx] (u) at (2,0) {};
    \node[vtx] (x) at (0,2) {};
    \node[vtx] (y) at (0,-2) {};
    
    \draw [gedge] (u) -- (v)  (u) -- (x)  (u) -- (y)  (v) -- (x) (v) -- (y);
    
    \draw[gedge] (x) -- ++(90:0.75);
    
    \draw[gedge] (y) -- ++(270:0.75);
    
    \draw[gedge] (v) -- ++(150:0.75);
    \draw[gedge] (v) -- ++(210:0.75);
    
    \draw[gedge] (u) -- ++(30:0.75);
    \draw[gedge] (u) -- ++(-30:0.75);
    \path[->]
    \chargepath{(v)}{dashedge}{}{}{(0,0.5)}
    \chargepath{(v)}{dashedge}{}{}{(0,-0.5)};

    \node at (0.5,0.5) {\footnotesize{$1.5$}};
    \node at (0.5,-0.5) {\footnotesize{$1.5$}};
    
    \node at (0,-4) {(R5)};
    \end{tikzpicture} \hspace{1 em}
     \begin{tikzpicture}[scale = 0.5]
    \node[vtx, label = $v$] (v) at (-2,0) {};
    \node[vtx] (u) at (2,0) {};
    \node[vtx] (x) at (0,2) {};
    \node[vtx] (y) at (0,-2) {};
    
    \draw [gedge] (u) -- (v)  (u) -- (x)  (u) -- (y)  (v) -- (x) (v) -- (y);
    
    \draw[gedge] (v) -- ++(150:0.75);
    \draw[gedge] (v) -- ++(210:0.75);
    
    \path[->]
    \chargepath{(v)}{dashedge}{}{}{(0,0.5)}
    \chargepath{(v)}{dashedge}{}{}{(0,-0.5)};

    \node at (0.4,0.5) {\footnotesize{$1$}};
    \node at (0.4,-0.5) {\footnotesize{$1$}};
    
    \node at (0,-4) {(R6)};
    \end{tikzpicture} \hspace{1 em}
    \begin{tikzpicture}[scale = 0.75]
    \baseConfigurationtriangleUnlabeled
    \draw (v) node[vtx,label=right:$v$]{};
    
    \draw[gedge] (v)--++(270:0.5);
    \draw[gedge] (v)--++(300:0.5);
    \draw[gedge] (v)--++(240:0.5);
    
    \draw[gedge] (u)--++(60:0.5);
    \draw[gedge] (u)--++(120:0.5);
    
    \draw[gedge] (w)--++(240:0.5);
    \draw[gedge] (w)--++(300:0.5);
    
    \path[->]
    \chargepath{(v)}{dashedge}{}{}{(2,0.5)}; 
    
    \node at (1.8,0.6) {\footnotesize{$1$}};
    \node at (2,-1) {(R7a)};
    \end{tikzpicture}\hspace{ 1 em}
    \begin{tikzpicture}[scale = 0.75]
    \baseConfigurationtriangleUnlabeled
    \draw (v) node[vtx,label=right:$v$]{};

    \draw[gedge] (v)--++(270:0.5);
    \draw[gedge] (v)--++(300:0.5);
    \draw[gedge] (v)--++(240:0.5);
    
    \draw[gedge] (u)--++(60:0.5);
    \draw[gedge] (u)--++(120:0.5);
    
    \draw[gedge] (w)--++(240:0.5);
    
    \path[->]
    \chargepath{(v)}{dashedge}{}{}{(2,0.5)}; 
    
    \node at (1.8,0.6) {\footnotesize{$2$}};
    \node at (2,-1) {(R7b)};
    \end{tikzpicture}\hspace{ 1 em}
    \begin{tikzpicture}[scale = 0.75]
    \baseConfigurationsquareUnlabeled
    \draw (v) node[vtx,label=left:$v$]{};
    
    \draw[gedge] (v)--++(240:0.5);
    \draw[gedge] (v)--++(260:0.5);
    \draw[gedge] (v)--++(280:0.5);
    \draw[gedge] (v)--++(300:0.5);

    \path[->]
    \chargepath{(v)}{dashedge}{}{}{(0,0)};
    
    \node at (0.2,0.2) {\footnotesize{$1$}};
    \node at (0,-2) {(R8)};
    \end{tikzpicture} \hspace{1 em}
    \begin{tikzpicture}[scale = 0.75]
    \baseConfigurationtriangleUnlabeled
    \draw (v) node[vtx,label=right:$v$]{};

    \draw[gedge] (v)--++(260:0.5);
    \draw[gedge] (v)--++(300:0.5);
    \draw[gedge] (v)--++(240:0.5);
    \draw[gedge] (v)--++(280:0.5);
    
    \path[->]
    \chargepath{(v)}{dashedge}{}{}{(2,0.5)}; 
    
    \node at (1.8,0.6) {\footnotesize{$2$}};
    \node at (2,-1) {(R8)};
    \end{tikzpicture}\hspace{ 1 em}
        \begin{tikzpicture}[scale = 0.5]
    \node[vtx, label = $v$] (v) at (-2,0) {};
    \node[vtx] (u) at (2,0) {};
    \node[vtx] (x) at (0,2) {};
    \node[vtx] (y) at (0,-2) {};
    
    \draw [gedge] (u) -- (v)  (u) -- (x)  (u) -- (y)  (v) -- (x) (v) -- (y);
    
    \draw[gedge] (x) -- ++(90:0.75);
    
    \draw[gedge] (y) -- ++(240:0.75);
    \draw[gedge] (y) -- ++(300:0.75);
    
    \draw[gedge] (v) -- ++(150:0.75);
    \draw[gedge] (v) -- ++(210:0.75);
    \draw[gedge] (v) -- ++(180:0.75);
    
    \path[->]
    \chargepath{(v)}{dashedge}{}{}{(0,0.5)}
    \chargepath{(v)}{dashedge}{}{}{(0,-0.5)};

    \node at (0.6,0.5) {\footnotesize{$1.75$}};
    \node at (0.6,-0.5) {\footnotesize{$1.75$}};
    
    \node at (0,-4) {(R9)};
    \end{tikzpicture} \hspace{1 em}
    \begin{tikzpicture}[scale = 0.5]
    \node[vtx, label = $v$] (v) at (-2,0) {};
    \node[vtx] (u) at (2,0) {};
    \node[vtx] (x) at (0,2) {};
    \node[vtx] (y) at (0,-2) {};
    
    \draw [gedge] (u) -- (v)  (u) -- (x)  (u) -- (y)  (v) -- (x) (v) -- (y);
    
    \draw[gedge] (x) -- ++(60:0.75);
    \draw[gedge] (x) -- ++(120:0.5);
    
    \draw[gedge] (y) -- ++(240:0.75);
    \draw[gedge] (y) -- ++(300:0.75);
    
    \draw[gedge] (v) -- ++(150:0.75);
    \draw[gedge] (v) -- ++(210:0.75);
    \draw[gedge] (v) -- ++(180:0.75);
    
    \path[->]
    \chargepath{(v)}{dashedge}{}{}{(0,0.5)}
    \chargepath{(v)}{dashedge}{}{}{(0,-0.5)};

    \node at (0.5,0.5) {\footnotesize{$1.5$}};
    \node at (0.5,-0.5) {\footnotesize{$1.5$}};
    
    \node at (0,-4) {(R10)};
    \end{tikzpicture} \hspace{1 em}    
    \begin{tikzpicture}[scale = 0.75]
    \baseConfigurationdiamond{}{}{}{}
    
    \draw[gedge] (u)--++(180:0.5) (u)--++(150:0.5) (u)--++(210:0.5);
    
    \path[->]
    \chargepath{(u)}{dashedge}{}{}{(0,0.5)}
    \chargepath{(u)}{dashedge}{}{}{(0,-0.5)};

    \node at (0.45,0.45) {\footnotesize{$1.25$}};
    \node at (0.45,-0.45) {\footnotesize{$1.25$}};
    
    \node at (0,-2) {(R11)};
    \end{tikzpicture} \hspace{2 em}
    \begin{tikzpicture}[scale = 0.75]
    \baseConfigurationdiamond{}{}{}{}
    
    \draw[gedge] (u)--++(170:0.5) (u)--++(150:0.5) (u)--++(210:0.5) (u)--++(190:0.5);
    
    \path[->]
    \chargepath{(u)}{dashedge}{}{}{(0,0.5)}
    \chargepath{(u)}{dashedge}{}{}{(0,-0.5)};

    \node at (0.45,0.45) {\footnotesize{$1.75$}};
    \node at (0.45,-0.45) {\footnotesize{$1.75$}};
    
    \node at (0,-2) {(R12)};
    \end{tikzpicture}
    \caption{Discharging Rules for Theorem~\ref{main567weak}}
\end{figure}
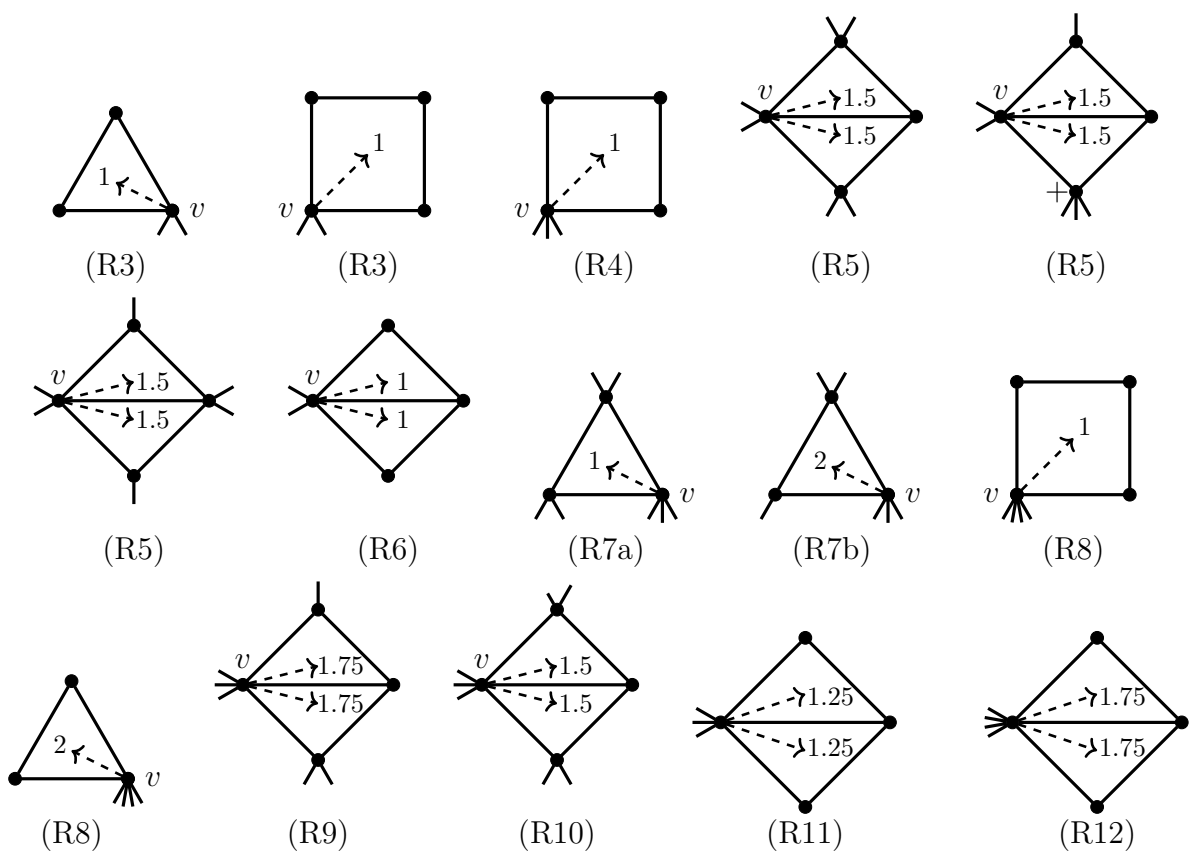

\begin{claim}\label{5.4}
The final charge of every vertex is nonnegative. 
\end{claim}
\begin{proof}
There are no vertices of degree less 
than 3, by (D1). The initial charge of a $3$-vertex is 0, and this does not change in the discharging process. 
A $4$-vertex $v$ has initial charge 2. It can be adjacent to at most two $4^-$-faces, or otherwise a cycle $C_k$ with $5 \le k \le 7$ is created. 
Therefore (R3) applies on $v$ at most twice and no other rules apply. Hence $v$ has a nonnegative final charge. 

Let $v$ be a $5$-vertex that is not a middle vertex of a diamond. Note that $v$ can be adjacent to at most two faces of size at most 4, or otherwise a cycle $C_k$ with $5 \le k \le 7$ is created. 
Thus, the initial charge of $v$ is 4,  
 and (R4) and (R7) are applied together at most twice, implying that $v$ has nonnegative final charge. 

Let $v$ be a $5$-vertex that is a middle vertex of a diamond $d$. Then $v$ is adjacent to at most one more face $f$ of size at most 4, and $f$ does not share any edge with $d$, or otherwise a cycle $C_k$ with $5 \le k \le 7$ is created.
If (R5) does not apply to $v$, then by (R4), (R6) and (R7), $v$ sends 1 to each of the two 3-faces in $d$ and at most 2 to $f$, leaving $v$ with final nonnegative charge.
Suppose (R5), where $v$ sends charge 3 to the faces in $d$, applies to $v$. If $v$ sends charge of at most 1 to $f$, then it has final nonnegative charge. So by (R4) and (R7a) we may assume that $f$ is a triangle $\{v,u,w\}$ with $d(u)=3$ (and  $d(w) \le 4$).
See Figure~\ref{fig:vtx} for an illustration. But then $G$ contains (D7), (D11), or (D8) as $d$ is $Dia(5-3,4,4)$,
$Dia(5-3,3,5^+)$, or
$Dia(5-5,3,3)$, respectively. Hence (R7b) does not apply to $v$ and the final charge is nonnegative.

\begin{figure}[h]
    \begin{center}
    \begin{tikzpicture}[scale = 0.5]
    \baseConfigurationdiamond{$v$}{}{}{}
    \node[vtx, label = left:$u$] at (-3,1) (a) {};
    \node[vtx, label = left:$w$] at (-3,-1) (b) {};
    
    \node at (-2.5,0) {$f$};
    
    \draw[gedge] (u)--(a) (u)--(b) (a)--(b);
    
    \draw[gedge] (x)--++(120:1) (x)--++(60:1) (y)--++(240:1) (y)--++(300:1);
    
    \draw[gedge] (a)--++(90:1) (b)--++(240:1) (b)--++(300:1);
    
    \draw[dashed, rounded corners=20pt] (-5,1.5)  -- (0,3.5) -- (2.5,0) --  (0,-3.5) -- (-2.5,-0.75) -- cycle;
    \end{tikzpicture} \hspace{ 1 cm }
        \begin{tikzpicture}[scale = 0.5]
    \baseConfigurationdiamond{$v$}{}{}{}
    \node[vtx, label = left:$u$] at (-3,1) (a) {};
    \node[vtx, label = left:$w$] at (-3,-1) (b) {};
    
    \node at (-2.5,0) {$f$};
    
    \draw[gedge] (u)--(a) (u)--(b) (a)--(b);
    
    \draw[gedge] (x)--++(90:1) (y)--++(240:1) (y)--++(300:1) (y)--++(270:1);
    
    \draw[gedge] (a)--++(90:1) (b)--++(240:1) (b)--++(300:1);
    
    \draw[dashed, rounded corners=20pt] (-5,1.5)  -- (0,3.5) -- (2.5,0) --  (0,-1.25) -- (-2.5,-0.75) -- cycle;
    
    \node at (0,-3.25) {};
    \end{tikzpicture}\hspace{ 1 cm }
        \begin{tikzpicture}[scale = 0.5]
    \baseConfigurationdiamond{$v$}{}{}{}
    \node[vtx, label = left:$u$] at (-3,1) (a) {};
    \node[vtx, label = left:$w$] at (-3,-1) (b) {};
    
    \node at (-2.6,0) {$f$};
    
    \draw[gedge] (u)--(a) (u)--(b) (a)--(b);
    
    \draw[gedge] (x)--++(90:1) (y)--++(270:1) (v)--++(30:1) (v)--++(-30:1);
    
    \draw[gedge] (a)--++(90:1) (b)--++(240:1) (b)--++(300:1);
    
    \draw[dashed, rounded corners=20pt] (-5,1.5)  -- (0,3.5) -- (3.5,0) --  (0,-3.5) -- (-2.5,-0.75) -- cycle;
    \end{tikzpicture}
    \end{center}
    \caption{Three cases in Claim \ref{5.4}.}
    \label{fig:vtx}
\end{figure}
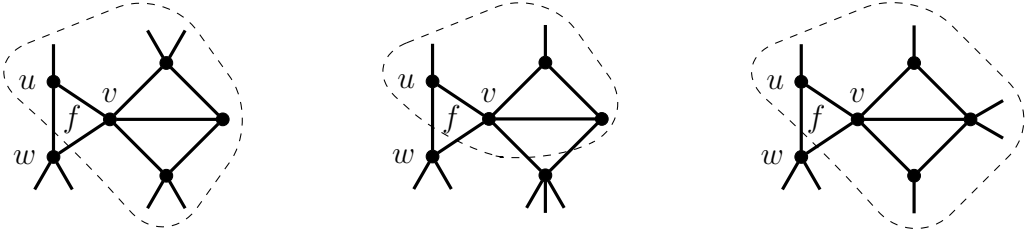

Let $v$ be a $6$-vertex that is the middle vertex of $k$ diamonds and it is adjacent to $m$ faces of size 3 or 4 that are not part of a diamond in which $v$ is a middle vertex. 
By Lemma~\ref{deglem}, $6 \ge 3k + 2m$.
Recall that $ch(v)=6$.
Suppose $k=2$, then $m=0$. 
By by (D12) and (D3), 
(R9) cannot apply twice and (R9) cannot apply at the same time as (R10).
Then by (R9)--(R11), the final charge of $v$ is at least $6-3.5-2.5 = 0$ or $6-3-3 = 0$. 
If $k=1$ and $m\le 1$, then by (R8)--(R11), the final charge of $v$ is at least $6-3.5-2 > 0$. 
Finally, if $k=0$ and $m\le 3$ then by (R8), the final charge of $v$ is at least $6-3\cdot 2 = 0$.

Let $v$ be a $7^+$-vertex that is the middle vertex of $k$ distinct diamonds, and $v$ is adjacent to $m$ faces of sizes 3 and 4 that are not part of a diamond in which $v$ is a middle vertex. Then by Lemma~\ref{deglem}, $\deg(v) \ge 3k + 2m$. By (R12) $v$ sends total weight of $3.5k$ to the $k$ diamonds in which $v$ is a middle vertex, and by (R8) it sends at most $2m$ to the other faces of size at most 4 it is adjacent to. Altogether, the final charge of $v$ is  at least $$2\deg(v) -6 - 3.5k -2m = 2\deg(v) -6 - (3k+2m) -k/2 \ge \deg(v) - 6 -\frac{1}{2}\cdot\Big\lfloor \frac{\deg(v)}{3} \Big\rfloor,  $$ where the last inequality follows from  Lemma \ref{deglem}, and $\deg(v) - 6 -\frac{1}{2}\cdot\Big\lfloor \frac{\deg(v)}{3} \Big\rfloor \ge 0$ whenever $\deg(v)\ge 7$. 
\end{proof}

\begin{claim}\label{claim 5}
The final charge of every face that is not contained in a diamond is nonnegative. 
\end{claim}
\begin{proof}
By (R1) and (R2), an $8^+$-face $f$ sends out a total charge of at most $\frac{\ell(f)}{4}$. Thus the final charge of $f$  is at least $\ell(f) - 6 - \frac{\ell(f)}{4}= \frac{3\ell(f)}{4} -6$ which is nonnegative  if $\ell(f)\ge 8$. 

Let $f$ be a $3$-face that is not part of any diamond. Then the faces sharing an edge with $f$ must be of size at least 8, since otherwise one of them is of size at most 4, which forces a diamond or a cycle $C_i$ with $5\le i\le 7$ together with $f$.
Hence (R1) applies three times with $f$ and $f$ has charge $-3 + \frac{3}{4} =-2.25$ after (R1).

By (D1) and (D2), one of the following holds (see Figure \ref{fig:claim 5}):
\begin{enumerate}
    \item[(1)] $f$ is $T(3,3,5^+)$
\item[(2)] $f$ is $T(3,4^+,4^+)$, or 
    \item[(3)] $f$ if $T(4^+,4^+,4^+)$.
\end{enumerate}

\begin{figure}[H]
    \centering
    \begin{tikzpicture}
    \baseConfigurationtriangleUnlabeled

    \draw[gedge] (w)--++(190:0.5) (w)--++(230:0.5) (w)--++(210:0.5) (v)--++(-30:0.5) (u)--++(90:0.5);
    \draw (2,-1) node  {$T(3,3,5^+)$};
    \end{tikzpicture} \hspace{ 2 em }
    \begin{tikzpicture}
    \baseConfigurationtriangleUnlabeled

    \draw[gedge] (w)--++(200:0.5) (w)--++(220:0.5) (v)--++(340:0.5) (v)--++(320:0.5) (u)--++(90:0.5);
        \draw (2,-1) node  {$T(3,4^+,4^+)$};
\end{tikzpicture} \hspace{2 em }
    \begin{tikzpicture}
    \baseConfigurationtriangleUnlabeled

    \draw[gedge] (w)--++(200:0.5) (w)--++(220:0.5) (v)--++(340:0.5) (v)--++(320:0.5) (u)--++(80:0.5)  (u)--++(100:0.5);
        \draw (2,-1) node  {$T(4^+,4^+,4^+)$};
\end{tikzpicture}
 \caption{Three possible triangles in Claim \ref{claim 5}}
    \label{fig:claim 5}
\end{figure}
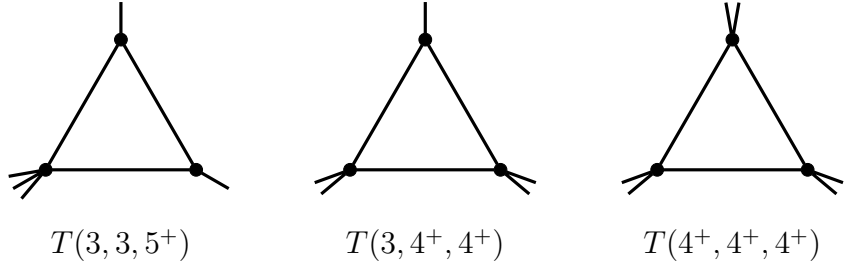

In case (1), (R2a) applies twice
 giving charge $\frac{4}{8}$ to $f$. In addition, (R7b) or (R8) applies and the final charge of $f$ is at least
$-3+\frac{3}{4}+\frac{1}{2}+2 \geq 0$. 

In case (2), (R2) applies once, giving charge $\frac{2}{8}$ to $f$. 
Rules (R3), (R7) and (R8) apply together twice with $f$, each time $f$ receives charge at least 1, and thus the final charge of $f$ is at least
$-3+\frac{3}{4}+\frac{1}{4}+2 \geq 0$. 

In case (3), rules (R3), (R7), and (R8) together apply three times to $f$ and thus the final charge of $f$ is at least $-3+\frac{3}{4} + 3 >0$. 

If $f$ is a 4-face, the faces sharing an edge with $f$ must be of size at least 8, since otherwise one of them is of size at most 4, which forces a cycle $C_i$ of size  $5\le i\le 7$ with $f$.
Hence (R1) applies four times with $f$ contributing charge $\frac{4}{4}$. 
By (D9) $f$ has at least two $4^+$-vertex.
Thus at least one of (R3), (R4), and (R8) applies to $f$, giving charge $1$ to $f$. 
Hence the final charge of $f$ is at least $-2 + 1 +1 \geq 0$.
\end{proof}

\begin{claim}\label{lem:diamond}
The final charge of every 3-face that is contained in a diamond is nonnegative. 
\end{claim}
\begin{proof}
In the light of (R13), we will consider the faces that form a diamond together in pairs
and show that as a pair, they receive sufficient charge.
Let $f$ and $g$ be 3-faces sharing an edge, i.e. they form a diamond.
Observe that in this case the other faces sharing edges with $f$ and $g$ must be of size at least $8$, for otherwise one of them is of size at most 4, which forces a cycle $C_i$ of size  $5\le i\le 7$ with $f$ and $g$. 
Therefore, (R1) applies twice to each $f$ and $g$ and $ch_1(f)=ch_1(g) = -3 + \frac{2}{4} = -2.5$.
Hence we aim to show that $f$ and $g$ together receive at least 5 more charge.
We denote the the vertices of $f$ by $u,v,x$ where $u,v$ are shared with $g$, and by $y$ the third vertex of $g$.

By symmetry, we assume that  $\deg(u) \ge \deg(v)$ and $\deg(y) \ge \deg(x)$. 
Note that by (D4), $\deg(u) \geq 4$ 
and by (D1) the degree of each of the other vertex is at least 3.

We split into cases based on the type of diamond $f$ and $g$ form.

\begin{itemize}
\item $Dia(4-3,\star,\star)$\\
By (D2) and (D6), $\deg(x) \ge 4$ and $\deg(y) \ge 5$. Hence we are in case
 $Dia(4-3,4^+,5^+)$.
For $u$, (R3) applies twice, for $x$  one of (R3), (R7b), or (R8) applies, and for $y$ one of (R7b), or (R8) applies.
Thus the charge $f$ and $g$ receive using these rules is at least $3\cdot 1 + 2 = 5$.
Hence the final charges of $f$ and $g$ are nonnegative.

\item $Dia(5-3,3,3)$,  $Dia(5-3,3,4)$\\
Reducible by (D9) and (D10).

\item $Dia(5-3,3,5^+)$\\
In this case, (R5) applies to $u$ and (R7b) or (R8) applies to $y$.
This gives charge $2\cdot 1.5 + 2 = 5$. 
Hence the final charges of $f$ and $g$ are nonnegative.

\item $Dia(5-3,4,4)$ \\
In this case, (R5) applies to $u$.
In addition (R3) applies to both $x$ and $y$.
This gives charge $2\cdot 1.5 + 1+1 = 5$. 
Hence the final charges of $f$ and $g$ are nonnegative.

\item  $Dia(5-3,4^+,5^+)$\\
In this case, (R6) applies to $u$.
In addition (R3), (R7b), or (R8) applies to $x$ and
(R7b), or (R8) to $y$.
This gives charge at least $2\cdot  1 + 1 +2 = 5$. 
Hence the final charges of $f$ and $g$ are nonnegative.

\item $Dia(6-3,3,3)$\\
Reducible by (D9).

\item $Dia(6-3,3,4)$\\
By (R9), $u$ contributes charge $3.5$ and by (R4), $y$ contributes charge $1$.
Let $z$ be a neighbor of $x$ that is not $u$ or $v$. By (D9), $\deg(z) \geq 4$
Hence the application of (R2b) around $x$ contributes charge $1/2$. 
This gives total charge $3.5+1+0.5 = 5$.
Hence the final charges of $f$ and $g$ are nonnegative.

\item $Dia(6-3,3,5^+)$ \\
By (R11), $u$ contributes charge $2.5$ and by (R7b) or (R8), $y$ contributes charge $2$.
Let $z$ be a neighbor of $x$ that is not $u$ or $v$. By (D9), $\deg(z) \geq 4$
Hence the application of (R2) around $x$ contributes charge $1/2$. 
This gives total charge $2.5+2+0.5 = 5$.
Hence the final charges of $f$ and $g$ are nonnegative.

\item $Dia(6-3,4,4)$ \\
By (R10), $u$ contributes charge $3$ and by (R3), $x$ and $y$ each contribute charge $1$.
This gives total charge $3+1+1 = 5$.
Hence the final charges of $f$ and $g$ are nonnegative.

\item $Dia(6-3,4^+,5^+)$\\
By (R11), $u$ contributes charge $2.5$, by (R3), (R7b) or (R8), $x$ contributes charge at least $1$,
and by (R7b) or (R8), $y$ contributes charge 2.
This gives total charge at least $2.5+1+2 = 5$.
Hence the final charges of $f$ and $g$ are nonnegative.

\item $Dia(7^+-3,3,3)$\\ 
Reducible by (D9).
 
\item $Dia(7^+-3,3,4^+)$\\
By (R12), $u$ contributes charge $3.5$, by (R3), (R7b) or (R8), $y$ contributes charge at least $1$.
Let $z$ be the neighbor of $x$ that is not $u$ or $v$. By (D9), $\deg(z) \geq 4$
Hence the application of (R2b) around $x$ contributes charge $1/2$. 
This gives total charge at least $3.5+1+0.5 = 5$.
Hence the final charges of $f$ and $g$ are nonnegative.

\item $Dia(7^+-3,4^+,4^+)$\\
By (R12), $u$ contributes charge $3.5$, by (R3), (R7b) or (R8), $x$ and $y$ each contribute charge at least $1$.
This gives total charge at least $3.5+1+1 = 5.5$.
Hence the final charges of $f$ and $g$ are nonnegative.

\item $Dia(4-4,\star,\star)$ and $Dia(5-4,\star,\star)$\\
By (D5), $\deg(y) \geq 4$.
By (R3) or (R6), $u$ and $v$ together contribute charge $4$, by (R3), $y$ each contributes charge $1$.
This gives total charge at least $4+1 = 5$.
Hence the final charges of $f$ and $g$ are nonnegative.

\item $Dia(6^+-4^+,\star,\star)$\\
By (R3), (R6), (R11), and (R12), $u$ and $v$ together contribute charge at least $1.25+1.25+1+1$.
If (R3), (R7), or (R8) applies to $y$, then the total charge is at least $5.5$. 
Hence we can assume the case $Dia(6^+-4^+,3,3)$. Then (R2a) or (R2b) applies at each $x$ and $y$ and the total contribution is at least $0.5$. 
This gives total charge at least $4.5+0.5 = 5$.
Hence the final charges of $f$ and $g$ are nonnegative.

\item $Dia(5-5,3,3)$\\
By (R5), $u$ and $v$ together contribute charge at least $4\times 1.5 = 6$.
Hence the final charges of $f$ and $g$ are nonnegative.

\item $Dia(5-5,3^+,4^+)$\\
By (R6), $u$ and $v$ together contribute charge at least $4\times 1 = 4$.
By (R3), (R7a), or (R8), $y$ contributes charge at least 1. 
This gives total charge at least $4+1 = 5$.
Hence the final charges of $f$ and $g$ are nonnegative.
\end{itemize}

This concludes the proof of Claim~\ref{lem:diamond}.
\end{proof}

Since all final charges are nonnegative, this concludes the proof of Lemma~\ref{lem:unavoidable Bk}.
\end{proof}

\bigskip
\noindent\textbf{Acknowledgements.~~}
The work was initiated at
Dagstuhl Seminar on ``Graph Colouring: from Structure to Algorithms''(19271).
Part of the work was carried out when T.~Masařík was visiting Iowa State University.
He thanks Steve Butler for kind hospitality.
We would like to thank Ilkyoo Choi for fruitful discussions at the beginning of this project.

\bibliographystyle{plainurl}
\bibliography{flexibility}

\end{document}